\documentclass[11pt,letterpaper]{amsart}
\usepackage{amssymb,latexsym}
\usepackage{amsfonts}
\usepackage{amsmath}
\usepackage{enumitem}
\usepackage{marvosym}
\usepackage{lipsum}
\usepackage[all]{xy}

\usepackage{color}

\newtheorem{claim}{Claim}[section]
\newtheorem{thm}{Theorem}[section]
\newtheorem{cor}[thm]{Corollary}
\newtheorem{lem}[thm]{Lemma}
\newtheorem{prop}[thm]{Proposition}
\newtheorem{defn}[thm]{Definition}
\newtheorem{rem}[thm]{\bf{Remark}}

\numberwithin{equation}{section}

\DeclareMathOperator \Vol{Vol}
\DeclareMathOperator \Id{Id}
\DeclareMathOperator \rank{rank}
\DeclareMathOperator \dif{d\!}
\DeclareMathOperator \real{\bf Re}

\DeclareMathOperator \Herm{Herm}
\DeclareMathOperator \End{End}
\DeclareMathOperator \cond{cond}
\DeclareMathOperator \tr{tr}
\DeclareMathOperator \diag{diag}

\def \pbp {\partial\bar\partial}

\begin{document}
\title{Long-time behavior of the Hermitian-Yang-Mills flow on non-K\"ahler manifolds}

\author[Zeng Chen]{Zeng Chen}
\address{University of Science and Technology of China, Hefei, 230026, P.R. China} \email{ustcchze@mail.ustc.edu.cn}
\author[Chao Li]{CHao Li}
\address{School of Mathematics and Statistics, Nanjing University of Science and Technology, Nanjing, 210094, P.R.China}\email{leecryst@mail.ustc.edu.cn}
\author[Chuanjing Zhang]{Chuanjing Zhang}
\address{School of Mathematics and Statistics, Ningbo University, Ningbo, 315211, P.R. China} \email{zhangchuanjing@nbu.edu.cn}
\author[Xi Zhang]{Xi Zhang}
\address{School of Mathematics and Statistics, Nanjing University of Science and Technology, Nanjing, 210094, P.R.China} \email[Corresponding author]{mathzx@njust.edu.cn}

\begin{abstract}
In this paper, we study the long-time behavior of the Hermitian-Yang-Mills flow over compact Hermitian manifolds. We obtain the  monotonicity of lower bound and upper bound of the eigenvalues of the mean curvature  along the Hermitian-Yang-Mills flow. In the Gauduchon case, we show that the eigenvalues of the mean curvature  converge to geometric invariants determined by the Harder-Narasimhan type. Furthermore, we generalize the Atiyah-Bott-Bando-Siu question to the non-K\"ahler case.
\end{abstract}

\thanks{ The research was supported by the National Key R and D Program of China 2020YFA0713100. The  authors are partially supported by NSF in China No.12141104, 12371062 and 12431004, the China Postdoctoral Science Foundation (No.2018M640583).}
\maketitle

\section{ Introduction}

Let $(M, \omega)$ be an $n$-dimensional compact Hermitian manifold, $(E, \overline{\partial}_E)$ a holomorphic vector bundle. In order to find a Hermitian-Einstein metric, Donaldson (\cite{DON85}) introduced the following evolution equation for the Hermitian metric  $H(t)$ on $(E, \overline{\partial}_E)$, which is called the Hermitian-Yang-Mills flow,
\begin{equation} \label{Flow}
H^{-1}(t)\frac{\partial H(t)}{\partial t}=-2(\sqrt{-1}\Lambda_{\omega } F_{H(t)}-\lambda \cdot \textmd{Id}_E),
\end{equation}
where $F_{H(t)}$ is the Chern curvature of $H(t)$, $\Lambda_\omega$ denotes the contraction of differential forms by $\omega$, and $\lambda$ is a constant. The long time existence and uniqueness of the solution to the above heat equation (\ref{Flow}) were proved by Donaldson (\cite{DON85}) for the K\"ahler manifold case, and by the fourth author (\cite{Z05}) for the general Hermitian manifold case.

When $\omega$ is K\"ahler, the long time behavior of the Hermitian-Yang-Mills flow (\ref{Flow}) is well-known: in the case of a stable bundle, the flow converges to a Hermitian-Einstein metric (\cite{DON85},\cite{DON87},\cite{SIM}); even in the general case, its convergence property at infinity is also clear due to the confirmation of the Atiyah-Bott-Bando-Siu conjecture (\cite{daskalopoulos1992topology},\cite{DW},\cite{jacob2016yang},\cite{Sib},\cite{LZZ2}). In this paper, we analyze the asymptotic behavior of the Hermitian-Yang-Mills flow (\ref{Flow}) on non-K\"ahler manifolds. Before stating our results, we would like to introduce some notations. Let $r=\rank E$. For any Hermitian metric $H$ on $E$, we denote the $r$ eigenvalues of the mean curvature $\sqrt{-1}\Lambda_{\omega} F_H$  by $\lambda_1(H,\omega)$, $\lambda_2(H,\omega)$, $\cdots$, $\lambda_r(H,\omega)$, sorted in the descending order.
Set
\begin{equation}\vec\lambda(H,\omega)=(\lambda_1(H,\omega),\lambda_2(H,\omega),\cdots,\lambda_r(H,\omega)),
\end{equation}
\begin{equation}
\lambda_L(H,\omega)=\lambda_r(H,\omega),\quad \lambda_U(H,\omega)=\lambda_1(H,\omega),
\end{equation}
\begin{equation}\hat{\lambda }_L(H,\omega)=\inf_M\ \lambda_L(H,\omega),\quad \hat{\lambda }_U(H,\omega)=\sup_M\ \lambda_U(H,\omega)
\end{equation}
and
\begin{equation}
\begin{split}
&\lambda_{mL}(H,\omega)=\frac{1}{\Vol(M, \omega)}\int_M\lambda_L(H,\omega)\frac{\omega^n}{n!},\\&\lambda_{mU}(H,\omega)=\frac{1}{\Vol(M, \omega)}\int_M\lambda_U(H,\omega)\frac{\omega^n}{n!}.
\end{split}
\end{equation}

We say a Hermitian metric $\omega $  is Gauduchon if it satisfies $\partial \bar{\partial} \omega^{n-1}=0$. Gauduchon (\cite{Gaud}) proved that for an arbitrary Hermitian metric $\hat{\omega }$ on a compact complex manifold, there exists a unique Gauduchon metric $\omega $, up to a positive constant, in the conformal class of $\hat{\omega }$. Our first result is the  monotonicity of $\hat{\lambda }_L(H(t),\omega)$ and $\hat{\lambda }_U(H(t),\omega)$ along the Hermitian-Yang-Mills flow.

\begin{thm}\label{thm2} Let $(M, \omega )$ be a compact Hermitian manifold, $(E, \bar{\partial }_{E})$ be a holomorphic vector bundle over $M$. Assume $H(t)$ is a smooth solution of the Hermitian-Yang-Mills flow (\ref{Flow}) on $(E, \bar{\partial }_{E})$. Then $\hat{\lambda }_L(H(t), \omega )$ is increasing and $\hat{\lambda }_U(H(t), \omega )$ is decreasing on $[0, \infty )$. If $\omega $ is Gauduchon, then $\lambda_{mL}(H(t), \omega )$ is increasing and $\lambda_{mU}(H(t), \omega )$ is decreasing. Furthermore, we have \begin{align}&\lim\limits_{t\rightarrow \infty}\lambda_{mL}(H(t), \omega )=\lim\limits_{t\rightarrow \infty}\hat{\lambda }_{L}(H(t), \omega ),\\ &\lim\limits_{t\rightarrow \infty}\lambda_{mU}(H(t), \omega )=\lim\limits_{t\rightarrow \infty}\hat{\lambda }_{U}(H(t), \omega ).\end{align}
\end{thm}

Our next result indicates the uniqueness of the limit of the mean curvature.
\begin{thm}\label{thmuni}
Suppose $H_{1}(t)$ and $H_{2}(t)$ are two smooth solutions of the Hermitian-Yang-Mills flow (\ref{Flow}). Then we have
\begin{equation}
\lim_{t\rightarrow \infty }\|\sqrt{-1}\Lambda_{\omega}(F_{H_{1}(t)}-F_{H_{2}(t)})\|_{L^{2}(H_{1}(t))}=0.
\end{equation}
\end{thm}

To see the geometric substance of the limit of the eigenvalues, we need to recall the Harder-Narasimhan type.
Let $(E, \bar{\partial }_{E})$ be an $r$-rank holomorphic vector bundle over a compact Gauduchon manifold $(M, \omega)$. It possesses a unique Harder–Narasimhan filtration $\{\mathcal{E}_{\alpha}\}_{\alpha=0}^l$ (here $\mathcal{E}_0= 0$ and $\mathcal{E}_l= E$), whose successive quotients $\mathcal{Q}_{\alpha }=\mathcal{E}_{\alpha}/\mathcal{E}_{\alpha-1} (1\leq \alpha \leq l)$ are torsion free and semistable (\cite{Bru}, \cite{Kobayashi}). From this filtration, one can construct a nonincreasing $r$-tuple of real numbers $\vec{\mu}_\omega(E )=(\mu_{1, \omega} , \cdots , \mu_{r, \omega})$, called the Harder-Narasimhan type of $(E, \bar{\partial }_{E})$, where $\mu_{i, \omega}= \mu_{\omega}(\mathcal{Q}_{\alpha})$ (i.e. the $\omega$-slope of the quotient $\mathcal{Q}_{\alpha}$) for $\rank(\mathcal{E}_{\alpha-1})+ 1 \leq i \leq \rank(\mathcal{E}_{\alpha})$. Equip $E$ with a Hermitian metric $K$. We have the associated orthogonal projections $\pi_{\alpha }^{K}: E\rightarrow E$ onto the subsheaves $\mathcal{E}_{\alpha }$ with respect to $K$. It is well-known that every
$\pi_{\alpha }^{K}$ is an $L_{1}^{2}$-bounded  Hermitian endomorphism. Define
\begin{equation}
\Phi_{\omega}^{HN} (E, K)=\Sigma_{\alpha=1}^{l}\mu_{\omega } (\mathcal{Q}_{\alpha })(\pi_{\alpha }^{K}-\pi_{\alpha-1}^{K}).
\end{equation}
This is an $L_{1}^{2}$-bounded Hermitian endomorphism, called the Harder-Narasimhan projection of  $(E, \bar{\partial}_{E})$.

In the Gauduchon case, the authors (\cite{LiZZ}) prove the existence of $L^{p}$-approximate critical Hermitian structure by the continuity method. Consider the following perturbed Hermitian-Einstein equation:
\begin{equation} \label{eq}
\sqrt{-1}\Lambda_{\omega } F_{H}-\lambda \cdot \textmd{Id}_E+\varepsilon \log (K^{-1}H)=0,
\end{equation}
where $K$ is any fixed background metric on $E$, and suppose $H_{\varepsilon}$ is a solution to (\ref{eq}).
\begin{thm}[\cite{LiZZ}] \label{theorem1}
Let $(M, \omega)$ be a compact Gauduchon manifold and $(E, \bar{\partial }_{E})$ a holomorphic vector bundle. Then there exists a sequence $\varepsilon_{i}\rightarrow 0$ such that
 \begin{equation}\label{eqthm1}
 \lim_{i\rightarrow \infty }\left\| \sqrt{-1}\Lambda_{\omega }F_{H_{\varepsilon_{i}}}-\frac{2\pi }{\Vol(M, \omega )}\Phi_{\omega }^{HN}(E, K)\right\|_{L^{p}(K)}=0
 \end{equation}
 for any $0<p<+\infty $.
\end{thm}

Then by  Theorem \ref{thm2}, Theorem \ref{thmuni} and Theorem \ref{theorem1}, we derive the limiting property of the eigenvalue functions of mean curvature along the Hermitian-Yang-Mills flow.

\begin{thm} \label{thm3}
Let $(M, \omega )$ be a compact Gauduchon manifold, $(E, \bar{\partial }_{E})$ be a holomorphic vector bundle of rank $r$ over $M$. Assume $H(t)$ is a smooth solution of the Hermitian-Yang-Mills flow (\ref{Flow}) on $(E, \bar{\partial }_{E})$. Then
\begin{equation}\lim_{t\rightarrow \infty }\left\|\vec\lambda(H(t), \omega )-\frac{2\pi }{\Vol(M, \omega)}\vec\mu_{\omega }(E)\right\|_{L^{2}}=0. \end{equation}
Specially, we have \begin{align}\label{Flow000}&\lim\limits_{t\rightarrow \infty}\lambda_{mL}(H(t), \omega )=\lim\limits_{t\rightarrow \infty}\hat{\lambda }_{L}(H(t), \omega )=\frac{2\pi }{\Vol(M, \omega )}\mu_{r, \omega },\\ &\lim\limits_{t\rightarrow \infty}\lambda_{mU}(H(t), \omega )=\lim\limits_{t\rightarrow \infty}\hat{\lambda }_{U}(H(t), \omega )=\frac{2\pi }{\Vol(M, \omega )}\mu_{1, \omega }.\end{align}
\end{thm}

\medskip

As a corollary to the above theorem, the existence theorem for special metrics on holomorphic vector bundles follows.

\begin{thm}\label{cor1}
Let $(M, \omega )$ be a compact Gauduchon manifold, $(E, \bar{\partial }_{E})$ be a holomorphic vector bundle over $M$.
Then for any $\delta>0$, there exists a Hermitian metric $H_{\delta}$ on $E$ such that
\begin{equation}\label{in01}(\frac{2\pi }{\Vol(M, \omega )}\mu_{L}(E, \omega )-\delta) \Id_{E}\leq \sqrt{-1}\Lambda_{\omega}F_{H_{\delta}}\leq (\frac{2\pi }{\Vol(M, \omega )}\mu_{U}(E, \omega )+\delta)\Id_{E}.
\end{equation}
\end{thm}

\begin{rem}\label{rem01}
If $(E, \bar{\partial }_{E})$ is $\omega$-semistable, then $\mu_{U}(E, \omega )=\mu_{L}(E, \omega )=\mu_{\omega}(E)$. So Theorem \ref{cor1} implies the existence of approximate  Hermitian-Einstein structure on semistable holomorphic bundles over compact Gauduchon manifolds, which was proved in \cite{NZ}.
\end{rem}

Using Theorem \ref{thm2} and Theorem \ref{thm3}, we can prove that

\begin{thm}\label{ample}
Let $(E, \bar{\partial}_{E})$ be an ample holomorphic vector bundle  of rank $2$ over a compact complex surface $M$. Then there must exist a Hermitian metric $H$ such that its second Chern form is positive, i.e. $c_{2}(E, H)>0$.
\end{thm}


Given a holomorphic vector bundle $(E, \overline{\partial}_E)$ with a Hermitian metric $H_{0}$ over a compact Hermtian manifold $(M, \omega )$,  we write $\mathcal{A}_{H_0}$ for the space of connections compatible with $H_{0}$, and $\mathcal{A}_{H_0}^{1,1}$ for the space of $H_0$-unitary integrable connections. The Hermitian-Yang-Mills flow (\ref{Flow}) with initial date $H_{0}$ is gauge equivalent to the following modified Yang-Mills flow with initial data ${A}_{0}^{}\in \mathcal{A}_{H_0}^{1,1}$:
     \begin{equation}
          \left\{\begin{aligned}\label{modified flow}
                                \frac{\partial A(t)}{\partial t} &=-\sqrt{-1}(\partial_{A(t)}-\overline{\partial}_{A(t)})\Lambda_{\omega}F_{A(t)},\\
                                A(0)&=A_0.\\
                 \end{aligned}
          \right.
    \end{equation}
Indeed, $A(t)=\sigma(t)(A_0)$, where $\sigma(t)$ is a complex gauge transformation satisfying $\sigma^{\ast_{H_0}}(t)\sigma(t)=H_0^{-1}H(t)$ and $H(t)$ is the smooth solution of the Hermitian-Yang-Mills flow (\ref{Flow}) with $H(0)=H_0$.

 When $(M, \omega )$ is a K\"ahler manifold, the above flow (\ref{modified flow}) is just the Yang-Mills flow which, as the gradient flow of the Yang-Mills functional, was introduced by  Atiyah-Bott in \cite{atiyah1983yang}. By Uhlenbeck compactness (\cite{uhlenbeck1982connections,uhlenbeck1986priori}), there exists a subsequence $t_i\to +\infty$ such that the sequence $A(t_{i})$ converges (modulo unitary gauge transformation) weakly in $L_1^2$-topology to a limiting connection $ A_{\infty}$ on a limiting bundle $(E_{\infty },H_{\infty})$ outside the  analytic bubbling set $\Sigma_{an}$. We call such a limit $A_{\infty}$ an Uhlenbeck limit. By the results in \cite{hong2004asymptotical}, one can show that the above convergence can be strengthened to be in $C_{loc}^{\infty}$-topology. Since $\sqrt{-1}\Lambda_{\omega}F_{A_{\infty}}$ is parallel, we have a holomorphic orthogonal decomposition on $M\setminus \Sigma_{an}$
    \begin{equation}
             (E_{\infty}, A_{\infty }, H_{\infty}) =\bigoplus_{i=1}^{l}(E_{\infty}^i, A_{\infty , i }, H_{\infty , i}),
    \end{equation} where every $H_{\infty , i}$ is a Hermitian-Einstein metric on the holomorphic vector bundle $E_{\infty}^i$. Moreover, $(E_{\infty}, \bar{\partial }_{A_{\infty }})$ extends over the whole $M$ as a reflexive sheaf. On the other hand, there is a double filtration called the Harder-Narasimhan-Seshadri filtration (abbr. HNS-filtration), whose successive quotients are actually stable not just semi-stable. Denote by $Gr_\omega^{HNS}(E,\overline{\partial}_E)$ the associated graded sheaf of the Harder-Narasimhan-Seshadri filtration of $(E,\overline{\partial}_E)$ with respect to $\omega$.
It is an interesting question whether $ (E_{\infty}, \bar{\partial }_{A_{\infty }})$ is isomorphic to $Gr_\omega^{HNS}(E,\overline{\partial}_E)^{\ast \ast}$ or not. This question was first raised by Atiyah-Bott (\cite{atiyah1983yang}) on Riemann surfaces, and has been proved by Daskalopoulos (\cite{daskalopoulos1992topology}). Bando and Siu (\cite{BS}) conjectured that it still holds for reflexive sheaf $\mathcal{E}$ over higher dimensional K\"ahler manifold. When the sheaf $\mathcal{E}$ is locally free, this question was confirmed by Daskalopoulos and Wentworth (\cite{DW}) for K\"ahler surfaces; by Jacob (\cite{jacob2016yang}) and Sibley (\cite{Sib}) in the higher dimensional K\"ahler case. The reflexive sheaf case was confirmed by Li, the third author and the fourth author (\cite{LZZ2}). Naturally, do these results hold in the non-K\"ahler case?

 In \cite{NZ2}, Nie and the fourth author studied the limiting behavior of the  modified Yang-Mills flow (\ref{modified flow}) under the assumption that the Hermitian metric $\omega $  is Gauduchon and astheno-K\"ahler. They proved:

  \begin{thm}[\cite{NZ2}]\label{thm1.1}
     Let $(E,{H}_{0}^{})$ be a Hermitian vector bundle over a compact Gauduchon and astheno-K$\ddot{a}$hler manifold $(M,\omega)$, and $A(t)$ be a global smooth solution of the modified Yang-Mills flow \eqref{modified flow} with smooth condition ${A}_{0}^{}\in \mathcal{A}_{H_0}^{1,1}$. Then
                                      $\\(1)$ for every sequence $t_k\to \infty$, there is a subsequence $t_{k_j}$ such that as $t_{k_j}\to \infty,$ $A(t_{k_j})$ converges (modulo unitary gauge transformations) to a connection $A_{\infty}$ satisfying
\begin{equation}\label{eq17}
 D_{A_{\infty}}\Lambda_{\omega } F_{A_{\infty}}=0
\end{equation}
on the Hermitian vector bundle $(E_{\infty}, H_{\infty})$ in $C^{\infty}_{loc}$-topology outside $\Sigma_{an},$ where $\Sigma_{an}$ is a closed set of Hausdorff codimension at least $4$.
                                      $\\(2)$ the limiting $({E}_{\infty}^{},{\overline{\partial}}_{{A}_{\infty}^{}}^{})$ can be extended to the whole M as a reflexive sheaf with a holomorphic orthogonal splitting
                                      \begin{equation}
             (E_{\infty}, A_{\infty }, H_{\infty}) =\bigoplus_{i=1}^{l}(E_{\infty}^i, A_{\infty , i }, H_{\infty , i}),
    \end{equation} where $l$ is the number of distinct eigenvalues of $\sqrt{-1}\Lambda_{\omega}F_{A_{\infty}}$ on $X\setminus \Sigma_{an}$ and  $H_{\infty , i}$ is an admissible Hermitian-Einstein metric on reflexive sheaf $E_{\infty}^i$.

     \end{thm}

Utilizing Theorem \ref{thm3}, we further investigate the structure of the Uhlenbeck limit, and give an affirmative answer to the above Atiyah-Bott-Bando-Siu question on some non-K\"ahler manifolds.
 \begin{thm}\label{thm1.9}
 Let $(E,{H}_{0}^{})$ be a Hermitian vector bundle on a compact Gauduchon and astheno-K\"ahler manifold $(M,\omega)$, and $A(t)$ be a global smooth solution of the modified Yang-Mills flow \eqref{modified flow} with smooth condition ${A}_{0}^{}\in \mathcal{A}_{H_0}^{1,1}$. Then
      the extended reflexive sheaf of any Uhlenbeck limiting $(E_{\infty},{\overline{\partial}}_{{A}_{\infty}^{}}^{})$ is isomorphic to the double dual of the graded sheaf associated to the HNS-filtration of the holomorphic vector bundle $(E, \bar{\partial }_{A_{0}})$, i.e.  we have
    \begin{equation}
            ({E}_{\infty}^{},{\overline{\partial}}_{{A}_{\infty}^{}}^{})\cong{Gr}_\omega^{HNS}{(E,{\overline{\partial}}_{{A}_{0}^{}}^{})}^{\ast\ast}.
    \end{equation}
    \end{thm}

    \medskip

     \begin{rem} The astheno-K\"ahler condition here is mainly used to control the Yang-Mills energy along the modified Yang-Mills flow \eqref{modified flow}, and then we can analyze the limiting behavior by Uhlenbeck's compactness. It is natural to ask that, in the absence of the astheno-K\"ahler condition, whether the Yang-Mills energy is uniformly bounded along the modified Yang-Mills flow.
     \end{rem}

Compared to the works on the Atiyah-Bott-Bando-Siu question in K\"ahler case ((\cite{DW,jacob2016yang,Sib,LZZ2}), the main difference is the proof that the HN-type of the limiting sheaf is equal to that of the initial holomorphic vector bundle $(E, \bar{\partial}_{A_0})$. Previous method (\cite[Lemma 4.3]{DW}) relies on the discreteness of slopes of the related sheaves on compact K\"ahler manifold. Instead, we give a new argument avoiding this restriction.

This paper is organized as follows. In Section 2,  we present some preliminaries and estimates which will be used in the proofs of Theorems \ref{thm2} and \ref{thmuni}. In Section 3, we analyze the asymptotic behavior of the Hermitian-Yang-Mills flow, and complete the proofs of Theorem \ref{thm2}, Theorem \ref{thmuni} and Theorem \ref{thm3}. In Section 4, we give some applications and show Theorem \ref{ample}. Section 5 is devoted to prove Theorem \ref{thm1.9}.

\medskip

{\bf  Acknowledgement:}  The research was supported by the National Key R and D Program of China 2020YFA0713100.
The  authors are partially supported by NSF in China No.12141104, 11801535, 11721101 and 11625106, the China Postdoctoral Science Foundation (No.2018M640583, 2018M642515) and the Fundamental Research Funds for the Central Universities.

\medskip

\section{Preliminary results}

\subsection{Stability and the Harder-Narasimhan filtration.}Assume $(M, \omega ) $ is an $n$-dimensional compact Gauduchon manifold and $\mathcal{F}$ is a torsion-free sheaf over $M$. The $\omega $-degree of $\mathcal{F}$ is given by
\begin{equation}
\deg_{\omega }(\mathcal{F} ):=\deg_{\omega }(\det(\mathcal{F}) )=\int_{M}c_{1}(\det(\mathcal{F}), H)\wedge \frac{\omega^{n-1}}{(n-1)!},
\end{equation}
where $H$ is an arbitrary Hermitian metric on $\det{\mathcal{F}}$. This is a well-defined real number independent of $H$ since $\omega^{n-1}$ is $\partial \bar{\partial }$-closed. We define the $\omega$-slope of $\mathcal{F}$ as
\begin{equation}
\mu_{\omega }(\mathcal{F} ):=\frac{\deg_{\omega }(\mathcal{F} )}{\rank (\mathcal{F} )}.
\end{equation}
 A holomorphic vector bundle $(E , \bar{\partial }_{E})$ is called $\omega$-stable (resp. $\omega$-semistable) if for every proper saturated subsheaf $\mathcal{S}\subset E$, there holds
\begin{equation}
\mu_{\omega}(\mathcal{S})<(resp. \leq)\mu_{\omega}(E).
\end{equation}
We say $H$  is
a Hermitian-Einstein metric on $(E, \bar{\partial }_{E} )$ if it satisfies
\begin{equation}
\sqrt{-1}\Lambda_{\omega} F_{H} =\lambda\cdot \mathrm{Id}_{E},
\end{equation}
where  $\lambda =\frac{2\pi }{\Vol(M, \omega )}\mu_{\omega}(E) $. The classical Donaldson-Uhlenbeck-Yau theorem (\cite{NS65,DON85,UY86}) states that, when $\omega$ is K\"ahler, the stability implies the existence of  Hermitian-Einstein metric. According to \cite{Bu,LT}, we know that the Donaldson-Uhlenbeck-Yau theorem is also valid for compact Gauduchon manifolds. There
are  many other interesting and important works related (\cite{ag,BS,Bi,bis0,bis,br,DW, HIT, HuLe,Jacob1,Kobayashi,JZ, LN2, LT,LT95, LZ, LZZ, LZZ2, Mo1, Mo2, Mo3,NR01,SIM,SIM2,WZ}, etc.).

Without the stability
assumption, we need the following result which also holds for an arbitrary torsion-free sheaf.
\begin{prop}[\cite{Kobayashi, Bru}]
Let $(E, \bar{\partial}_{E})$ be a holomorphic vector bundle over a compact Gauduchon manifold $(M, \omega)$. Then there is a unique filtration of $E$ by subsheaves
\begin{equation}\label{HNS01}
0=\mathcal{E}_{0}\subset \mathcal{E}_{1}\subset \cdots \subset \mathcal{E}_{l}=E
\end{equation}
such that every quotient sheaf $\mathcal{Q}_{\alpha}=\mathcal{E}_{\alpha}/\mathcal{E}_{\alpha-1}$ is torsion-free and $\omega$-semistable, and $\mu_{\omega } (\mathcal{Q}_{\alpha})>\mu_{\omega } (\mathcal{Q}_{\alpha+1})$. This filtration is called the Harder-Narasimhan filtration (abbr. HN-filtration) of $(E, \bar{\partial}_{E})$.
\end{prop}
If $\rank(E)= r$, from the Harder-Narasimhan filtration, we have a nonincreasing $r$-tuple of numbers
\begin{equation}
\vec{\mu}_\omega(E )=(\mu_{1, \omega}(E), \cdots, \mu_{r, \omega}(E)),
\end{equation}
where $\mu_{i, \omega}(E)=\mu_{\omega} (\mathcal{Q}_{\alpha})$, for $\rank(\mathcal{E}_{\alpha -1})+1\leq i\leq \rank(\mathcal{E}_{\alpha})$. We call $\vec{\mu}_\omega(E)$ the Harder-Narasimhan type (abbr. HN-type) of $(E, \bar{\partial}_{E})$.

There is an analogous filtration for semistable sheaves.
\begin{prop}[\cite{Kobayashi, Bru}]
Let $\mathcal{Q}$ be a semi-stable torsion free sheaf over a compact Gauduchon manifold $(M, \omega)$. Then there is a filtration
\begin{equation}
0=\mathcal{F}_{0}\subset\mathcal{F}_{1}\subset\cdots\mathcal{F}_{k}=\mathcal{Q},
\end{equation}
called a Seshadri filtration of $\mathcal{Q}$, such that the quotients $\mathcal{F}_{\alpha}/\mathcal{F}_{\alpha-1}$ are torsion-free and $\omega$-stable, $\mu_{\omega}(\mathcal{F}_{\alpha}/\mathcal{F}_{\alpha-1})=\mu_{\omega}(\mathcal{Q})$ for each $\alpha$. While such a filtration may not be unique, the associated graded object $Gr_\omega^{S}(\mathcal{Q},\overline{\partial}_{\mathcal{Q}})=\oplus_{\alpha=1}^{k}\mathcal{F}_{\alpha}/\mathcal{F}_{\alpha-1}$ is uniquely determined by $\mathcal{Q}$ up to an isomorphism.
\end{prop}
From this, one can get a refinement of the Harder–Narasimhan filtration by decomposing the semi-stable factors into filtrations whose quotients are actually stable.
\begin{prop}[\cite{Kobayashi, Bru}]
Let $(E, \bar{\partial}_{E})$ be a holomorphic vector bundle over a compact Gauduchon manifold $(M, \omega)$. Then there is a double filtration $\{E_{\alpha, \beta}\}$, called the Harder-Narasimhan-Seshadri filtration (abbr. HNS-filtration) of $(E, \bar{\partial}_{E})$, with the following properties: if $0=\mathcal{E}_{0}\subset \mathcal{E}_{1}\subset \cdots \subset \mathcal{E}_{l}=E$ is the Harder-Narasimhan filtration, then
\begin{equation}
\mathcal{E}_{\alpha-1}=\mathcal{E}_{\alpha, 0}\subset \mathcal{E}_{\alpha, 1}\subset \cdots \subset \mathcal{E}_{\alpha, k_\alpha}=\mathcal{E}_{\alpha},
\end{equation}
where the successive quotients $\mathcal{Q}_{\alpha, \beta}=\mathcal{E}_{\alpha, \beta}/\mathcal{E}_{\alpha, \beta-1}$ are torsion-free and $\omega$-stable. Moreover, $\mu_{\omega}(\mathcal{Q}_{\alpha, \beta})=\mu_{\omega}(\mathcal{Q}_{\alpha, \beta+1})$ and $\mu_{\omega}(\mathcal{Q}_{\alpha, \beta})> \mu_{\omega}(\mathcal{Q}_{\alpha+1, \beta})$.
The associated graded object
\begin{equation}
Gr_\omega^{HNS}(E,\overline{\partial}_E)=\oplus_{\alpha=1}^{l}\oplus_{\beta=1}^{k_\alpha}\mathcal{Q}_{\alpha, \beta}
\end{equation}
uniquely determined by the isomorphism class of $E$.
\end{prop}

\subsection{The Hermitian-Yang-Mills flow}Let $(M, \omega )$ be a compact Hermitian manifold of complex dimension $n$, and $(E, \bar{\partial}_{E})$ be a rank $r$ holomorphic vector
bundle over $M$. If $H$ is a Hermitian metric on $E$, we denote the Chern connection by $D_{H}$ and the curvature form by $F_{H}$. Suppose $K$ is another Hermitian metric and set $h=K^{-1}H$. We have the following identities
\begin{equation}\label{id1}
\begin{split}
& \partial _{H}-\partial_{K} =h^{-1}\partial_{K}h,\\
& F_{H}-F_{K}=\bar{\partial }_{E} (h^{-1}\partial_{K} h) ,\\
\end{split}
\end{equation}
where $\partial _{H}=D^{1, 0}_{H}$ and $\partial _{K}=D^{1, 0}_{K}$ are the $(1, 0)$-parts of $D_H$ and $D_K$, respectively.

Here we recall some results on the Hermitian-Yang-Mills flow (\ref{Flow}). With the notation $h(t)=K^{-1}H(t)$, equation (\ref{Flow}) can be rewritten as
\begin{equation}
\frac{\partial h(t)}{\partial
t}=-2\sqrt{-1}h(t)\Lambda_{\omega}(F_{K}+\bar{\partial
}_{E}(h^{-1}(t)\partial_{K}h(t)))+2\lambda h(t).
\end{equation}
By (\ref{Flow}), the evolution equation of $\sqrt{-1}\Lambda_{\omega}F_{H(t)}$ is
\begin{equation}\label{evlK}
\frac{\partial \sqrt{-1}\Lambda_{\omega}F_{H(t)}}{\partial t}=-2\sqrt{-1}\Lambda_\omega\bar\partial_E\partial_{H(t)} \sqrt{-1}\Lambda_{\omega}F_{H(t)}.
\end{equation}

\begin{prop}[\cite{DON85,Z05}]
The Hermitian-Yang-Mills flow (\ref{Flow}) with initial data $H_{0}$ must have a unique solution $H(t)$ which exists for $0\leq t < +\infty$.
\end{prop}

When no confusion can arise, we sometimes omit the parameter $t$ in the sequel for simplicity. Set $\Phi (H)=\sqrt{-1}\Lambda_{\omega}F_{H} -\lambda \cdot \textmd{Id}_E$, then we have

\begin{prop}[{\cite[Proposition 2.2]{Z05}}]\label{Flow002}
Let $H(t)$ be a solution of Hermitian-Yang-Mills flow (\ref{Flow}).
Then
\begin{equation}\label{eqmc01}
(2\sqrt{-1}\Lambda_{\omega }\partial \bar{\partial }-\frac{\partial }{\partial t})\tr \Phi (H(t))=0
\end{equation}
and
\begin{equation}\label{eqmc}
(2\sqrt{-1}\Lambda_{\omega }\partial \bar{\partial }-\frac{\partial }{\partial t})|\Phi (H(t))|_{H(t)}^{2}=2|D_{H(t)}\Phi (H(t))|_{H(t)}^{2}.
\end{equation}
Moreover, if the initial Hermitian metric $H_{0}$ satisfies $\tr (\sqrt{-1}\Lambda_{\omega } F_{H_{0}}-\lambda \cdot \Id_E)=0$, then
\begin{equation}
\tr F_{H(t)} =\tr F_{H_{0}}
\end{equation}  and $\det (H_{0}^{-1}H(t))=1$.
\end{prop}

Since $\tr \sqrt{-1}\Lambda_\omega F_H=\tr \Phi(H)+r\lambda$ and $|\sqrt{-1}\Lambda_\omega F_H|_H^2=|\Phi(H)|_H^2-2\lambda\tr\Phi(H)+r\lambda^2$, both (\ref{eqmc01}) and (\ref{eqmc}) still hold if we replace $\Phi(H)$ by $\sqrt{-1}\Lambda_\omega F_H$. The maximum principle yields
\begin{equation}\label{bdPhi}
\sup_M|\Phi (H)|_H^2(t)\leq \sup_M|\Phi (H)|_H^2(0)
\end{equation}
and
\begin{equation}\label{bdK}
\sup_M|\sqrt{-1}\Lambda_{\omega}F_{H}|_H^2(t)\leq \sup_M|\sqrt{-1}\Lambda_{\omega}F_{H}|_H^2(0).
\end{equation}

For simplicity, we use the same notation $\bar\partial$ to stand for both the ordinary operator $\bar\partial$  and holomorphic structure $\bar\partial _E$, and set $\theta_H=\sqrt{-1}\Lambda_\omega F_H$. Like in \cite[Chapter VII]{DemAgbook}, we denote by $\tau$ the type $(1,0)$ operator of zero order defined by $\tau=[\Lambda_\omega,\partial\omega]$. One has the fact that
\begin{equation}
[\sqrt{-1}\Lambda_\omega,\bar\partial]=\partial^{*_H}+\tau^*,\qquad [\sqrt{-1}\Lambda_\omega,\partial_H]=-(\bar\partial^{*_H}+\bar\tau^*),
\end{equation}
where $\partial^{*_H}$ and $\bar\partial^{*_H}$ are the formal adjoint operators of $\partial_H$ and $\bar\partial$ with respect to $H$ respectively.

The following lemma is well-known.
\begin{lem}\label{estK}
Assume that $\omega$ is Gauduchon. Let $H(t)$ be a solution of the Hermitian-Yang-Mills flow (\ref{Flow}), then
\begin{equation}\label{intDK}
\int_0^{\infty}\int_M |D_H\theta_H|_H^2(t) \frac{\omega^n}{n!}dt\leq \frac{1}{2}\int_M |\theta_H|_H^2(0) \frac{\omega^n}{n!}
\end{equation}
and
\begin{equation}\label{limDK}
\lim_{t\rightarrow \infty}\int_M |D_H\theta_H|_H^2(t)\frac{\omega^n}{n!}=0.
\end{equation}
\end{lem}
For the convenience of the readers, we give a proof here.
\begin{proof}
As mentioned above, we see
\begin{equation}\label{eqmc000}
\left({\frac{\partial }{\partial t}}-2\sqrt{-1}\Lambda_\omega\pbp\right)|\theta_H|_H^2(t)=-2|D_H \theta_H|_H^2(t).
\end{equation}
Then (\ref{intDK}) follows directly.

In order to show (\ref{limDK}), we need to estimate the growth of $\int_M |D_H\theta_H|_H^2(t) \frac{\omega^n}{n!}$.
Based on (\ref{evlK}),  one can find
\begin{equation}
\begin{split}
\frac{\dif\,}{\dif t}\int_M|\bar\partial \theta_H|^2_{H,\omega}\frac{\omega^n}{n!}
=&2\int_M\langle [\bar\partial \theta_H,\theta_H], \bar\partial\theta_H \rangle_{H,\omega}\frac{\omega^n}{n!}\\
&-4\real\int_M\langle \bar\partial\sqrt{-1}\Lambda_\omega \bar\partial\partial_{H}\theta_H,\bar\partial \theta_H \rangle_{H,\omega}\frac{\omega^n}{n!}.
\end{split}
\end{equation}
For the last integral, we have
\begin{equation}
\quad\int_M\langle \bar\partial\sqrt{-1}\Lambda_\omega \bar\partial\partial_{H}\theta_H, \bar\partial \theta_H \rangle_{H,\omega}\frac{\omega^n}{n!}=\int_M\langle \sqrt{-1}\Lambda_\omega \bar\partial\partial_{H}\theta_H, \bar\partial^{*_H}\bar\partial \theta_H\rangle_{H}\frac{\omega^n}{n!}
\end{equation}
and
\begin{equation}\begin{aligned}
\bar\partial^{*_H}\bar\partial \theta_H &=-([\sqrt{-1}\Lambda_{\omega},\partial_{H}]+\bar\tau^*)\bar\partial \theta_H=-\sqrt{-1}\Lambda_{\omega}\partial_{H}\bar\partial \theta_H-\bar\tau^*\bar\partial \theta_H\\
&=\sqrt{-1}\Lambda_\omega \left(\bar\partial\partial_{H}\theta_H-[F_H,\theta_H]\right)-\bar\tau^*\bar\partial \theta_H\\
&=\sqrt{-1}\Lambda_\omega\bar\partial\partial_{H}\theta_H-\bar\tau^*\bar\partial \theta_H.
\end{aligned}\end{equation}
Then it follows that
\begin{equation}\begin{aligned}
&\quad\real\int_M\langle \bar\partial\sqrt{-1}\Lambda_\omega \bar\partial\partial_{H}\theta_H, \bar\partial \theta_H \rangle_{H,\omega}\frac{\omega^n}{n!}\\
&=\real\int_M\langle \sqrt{-1}\Lambda_\omega\bar\partial\partial_{H}\theta_H, \sqrt{-1}\Lambda_\omega\bar\partial\partial_{H}\theta_H-\bar\tau^*\bar\partial \theta_H \rangle_{H}\frac{\omega^n}{n!}\\
&\geq \int_M|\sqrt{-1}\Lambda_\omega\bar\partial\partial_{H}\theta_H|_H^2 \frac{\omega^n}{n!}-\int_M |\sqrt{-1}\Lambda_\omega\bar\partial\partial_{H}\theta_H|_H |\bar\tau^*\bar\partial \theta_H|_H\frac{\omega^n}{n!}\\
&\geq-\frac{1}{4}\int_M|\bar\tau^*\bar\partial \theta_H|_H^2 \frac{\omega^n}{n!}.
\end{aligned}\end{equation}
Therefore,  there holds
\begin{equation}
\begin{aligned}
\quad\frac{\dif\,}{\dif t}\int_M|\bar\partial \theta_H|^2_{H,\omega}\frac{\omega^n}{n!}&\leq 2\int_M\langle [\bar\partial \theta_H,\theta_H], \bar\partial \theta_H \rangle_{H,\omega}\frac{\omega^n}{n!}+\int_M|\bar\tau^*|^2_{\omega}|\bar\partial \theta_H|_{H, \omega}^2 \frac{\omega^n}{n!}\\
&\leq C\int_M|\bar\partial \theta_H|^2_{H,\omega}\frac{\omega^n}{n!},
\end{aligned}
\end{equation}
where $C$ is a constant depending only on $\sup_M|\theta_H|_H (0)$ and $\sup_M|d\omega|_{\omega}$.
Since $|D_H\theta_H|_{H, \omega}^2=2|\bar\partial \theta_H|_{H, \omega}^2$, we get
\begin{equation}\label{evlDK}
\frac{\dif\,}{\dif t}\int_M|D_H\theta_H|^2_{H,\omega}\frac{\omega^n}{n!}\leq C\int_M|D_H\theta_H|^2_{H,\omega}\frac{\omega^n}{n!}.
\end{equation}
This together with (\ref{intDK}) gives (\ref{limDK}).
\end{proof}

\medskip

For later use, we give the following result on the long-time behavior of weak super-solutions of parabolic equations.

\begin{lem}\label{ss}
Let $u\in\mathcal C(M\times [0,\infty))$ be a function satisfying
\begin{equation}\bigg(\frac{\partial}{\partial t}-2\sqrt{-1}\Lambda_\omega\pbp\bigg) u\geq 0
\end{equation}
in the viscosity sense. Then $\inf_M u(\cdot,t)$ is monotone increasing on $[0,\infty)$.
If  in addition $\omega$ is Gauduchon, then the mean value at time $t$
\begin{equation}
\mu_m(t)\triangleq\frac{1}{\Vol(M,\omega)}\int_M u(\cdot,t)\frac{\omega^n}{n!}
\end{equation}
is monotone increasing on $[0,\infty)$ and
\begin{equation}
\lim_{t\rightarrow\infty}\inf_M u(\cdot,t)=\lim_{t\rightarrow\infty}\mu_m(t).
\end{equation}
\end{lem}
The viscosity super-solution (see e.g. \cite{CIL92}) is a very weak notion that admits the comparison principle, and it coincides with the distribution super-solution in the continuous case (\cite{Lions83,Ish95}). In fact, the weak super-solutions  considered in Section 3 belong to a stronger notion introduced in \cite{Cal58,Dod83}.
\begin{proof}
The monotone increasing property of $\inf_M u(\cdot,t)$ comes directly from the comparison principle for the viscosity super-solution.

Assume that $\omega$ is Gauduchon. Given a $t_0\geq 0$, we suppose $v$ is the continuous solution of the equation
\begin{equation}\left\{\begin{split}
\bigg(\frac{\partial}{\partial t}-2\sqrt{-1}\Lambda_\omega\pbp\bigg) v=0,&\quad \textrm{on }M\times(t_0,\infty),\\
v=u,&\quad \textrm{on }M\times\{t_0\},
\end{split}
\right.
\end{equation}
then $v$ is smooth on $M\times(t_0,\infty)$ and for any $t\geq t_0$,
\begin{equation}
\frac{1}{\Vol(M,\omega)}\int_M v(\cdot,t)\frac{\omega^n}{n!}=\frac{1}{\Vol(M,\omega)}\int_M v(\cdot,t_0)\frac{\omega^n}{n!}=\mu_m(t_0).
\end{equation}
 Clearly $(u-v)$ is a super-solution in the viscosity sense and vanishes on $M\times\{t_0\}$. By the comparison principle, we have $(u-v)\geq 0$ on $M\times [t_0,\infty)$ and consequently for any $t> t_0$,
\begin{equation}
\mu_m(t)\geq \frac{1}{\Vol(M,\omega)}\int_M v(\cdot,t)\frac{\omega^n}{n!}=\mu_m(t_0).
\end{equation}
This implies $\mu_m(t)$ is monotone increasing on $[0,\infty)$.

At the same time, since $\big(v-\inf_Mv(\cdot, t_0)\big)$ is a nonnegative solution, the parabolic Harnack inequality indicates
\begin{equation}
\sup_M v(\cdot,t_0+1)-\inf_M v(\cdot,t_0)\leq C\bigg(\inf_M v(\cdot, t_0+1)-\inf_M v(\cdot,t_0)\bigg),
\end{equation}
where $C\geq 1$ is a constant depending only on $(M,\omega)$. Then
\begin{equation}
0\leq \mu_m(t_0)-\inf_M u(\cdot, t_0)\leq C\bigg(\inf_M u(\cdot, t_0+1)-\inf_M u(\cdot,t_0)\bigg).
\end{equation}
Accordingly we deduce $\lim_{t\rightarrow \infty}\mu_m(t)=\lim_{t\rightarrow \infty}\inf_M u(\cdot,t)$.
\end{proof}

\section{Proofs of Theorem \ref{thm2}, \ref{thmuni} and \ref{thm3}}

Let $(M, \omega )$ be an $n$-dimensional compact Hermitian manifold, $(E, \bar{\partial }_{E})$ a rank $r$ holomorphic vector bundle over $M$ and $H$ a Hermitian metric on $E$. Set
\begin{equation*}
	\Herm(E, H)=\{ \theta\in \Gamma(\End(E)) \;  | \;  \theta^{\ast H}=\theta\},
\end{equation*}	
which is the real vector space of smooth $H$-selfadjoint endomorphisms of  $E$, and define $\Herm^{+}(E, H)\subset \Herm(E, H)$ to be the open subset of  positive definite ones.
As mentioned in Introduction, $\lambda_1(H,\omega)$, $\lambda_2(H,\omega)$, $\cdots$, $\lambda_r(H,\omega)$ are the eigenvalues of $\theta_H=\sqrt{-1}\Lambda_\omega F_H$, sorted in the descending order. For $1 \leq k \leq r$, write
\begin{equation}
\lambda_{L,k}(H,\omega)=\sum_{i=1}^k\lambda_{r-i+1}(H,\omega),\qquad \lambda_{U,k}(H,\omega)=\sum_{i=1}^k\lambda_{i}(H,\omega),
\end{equation}
and
\begin{equation}\hat\lambda_{L,k}(H,\omega)=\inf_M \lambda_{L,k}(H,\omega),\qquad \hat\lambda_{U,k}(H,\omega)=\sup_M \lambda_{U,k}(H,\omega).
\end{equation}
If additionally, $\omega$ is Gauduchon, we set
\begin{equation}\begin{aligned}
&\lambda_{mL,k}(H,\omega)=\frac{1}{\Vol(M,\omega)}\int_M\lambda_{L,k}(H,\omega)\frac{\omega^n}{n!},\\
&\lambda_{mU,k}(H,\omega)=\frac{1}{\Vol(M,\omega)}\int_M\lambda_{U,k}(H,\omega)\frac{\omega^n}{n!}.
\end{aligned}\end{equation}

We start with the convergence of  eigenvalues of the mean curvature along the Hermitian-Yang-Mills flow. Let $H(t)$ be a smooth solution of (\ref{Flow}). We write for short
\begin{equation}
\begin{split}
&\lambda_{L,k}(\cdot,t)=\lambda_{L,k}(H(t),\omega)(\cdot), \qquad\ \  \lambda_{U,k}(\cdot,t)=\lambda_{U,k}(H(t),\omega)(\cdot),\\
&\hat\lambda_{L,k}(t)=\hat\lambda_{L,k}(H(t),\omega), \qquad\qquad\ \hat\lambda_{U,k}(t)=\hat\lambda_{U,k}(H(t),\omega),\\
&\lambda_{mL,k}(t)=\lambda_{mL,k}(H(t),\omega), \qquad\quad \lambda_{mU,k}(t)=\lambda_{mU,k}(H(t),\omega).
\end{split}
\end{equation}

We derive the following generalization of Theorem \ref{thm2}.

\begin{thm}\label{thm2'}
Let $1 \leq k \leq r$. Then $\hat\lambda_{L,k}(t)$ is monotone increasing and $\hat\lambda_{U,k}(t)$ is monotone decreasing on $[0,\infty)$.
If additionally $\omega$ is Gauduchon, then $\lambda_{mL,k}(t)$ is increasing and $\lambda_{mU,k}(t)$ is decreasing. Furthermore, we have
\begin{equation}
\lim\limits_{t\rightarrow \infty}\lambda_{mL,k}(t)=\lim\limits_{t\rightarrow \infty}\hat\lambda_{L,k}(t),\qquad
\lim\limits_{t\rightarrow \infty}\lambda_{mU,k}(t)=\lim\limits_{t\rightarrow \infty}\hat\lambda_{U,k}(t).
\end{equation}
\end{thm}

By  Lemma \ref{ss}, the proof of Theorem \ref{thm2'} can be reduced to prove that
\begin{equation}
({\frac{\partial}{\partial t}}-2\sqrt{-1}\Lambda_\omega \pbp)\lambda_{L,k}\geq 0,\qquad ({\frac{\partial}{\partial t}}-2\sqrt{-1}\Lambda_\omega\pbp)\lambda_{U,k}\leq 0
\end{equation}
in the viscosity sense. Actually, the latter is a simple corollary of the following technical lemma.
\begin{lem}\label{klm1'}
Let $1 \leq k \leq r$.
For any $(p_0,t_0)\in M\times[0,\infty)$, we can find an open neighborhood $U$ of $p_0$, and smooth functions $f_{L,k}, f_{U,k}$ on $U\times[0,\infty)$, such that
\begin{itemize}
\item[1)] $f_{L,k}\geq \lambda_{L,k}$ on $U\times[0,\infty)$ and $f_{L,k}(p_0,t_0)=\lambda_{L,k}(p_0,t_0)$;
\item[2)] $\left(\frac{\partial}{\partial t}-2\sqrt{-1}\Lambda_{\omega}\pbp\right)f_{L,k}(p_0,t_0)=0$;
\item[3)] $f_{U,k}\leq \lambda_{U,k}$ on $U\times[0,\infty)$ and $f_{U,k}(p_0,t_0)=\lambda_{U,k}(p_0,t_0)$;
\item[4)] $\left(\frac{\partial}{\partial t}-2\sqrt{-1}\Lambda_{\omega}\pbp\right)f_{U,k}(p_0,t_0)=0$.
\end{itemize}
\end{lem}

\begin{proof}
It suffices to show the construction of $f_{L,k}$, the construction of $f_{U,k}$ is similar.
Certainly one can find an othonormal basis $\{u_{\alpha}\}_{\alpha=1}^r$ of $(E|_{p_0},H(t_0)|_{p_{0}})$ such that
\begin{equation}\langle\sqrt{-1}\Lambda_{\omega}F_{H(t_0)}u_{\alpha},u_{\beta}\rangle_{H(t_0)}(p_0)=\lambda_{\alpha}(H(t_0),\omega)(p_0)\delta_{\alpha\beta}.\end{equation}
Choose a holomorphic frame field $\{e_{\alpha}\}_{\alpha=1}^r$ on some open neighborhood $U$ of $p_0$ such that
\begin{itemize}
\item[1)] $e_{\alpha}(p_0)=u_{\alpha}$ for $1 \leq \alpha \leq r$;
\item[2)] $dH_{\bar\beta\alpha}\big|_{(p_0,t_0)}=0$, where $H_{\bar\beta\alpha}=\langle e_{\alpha},e_{\beta}\rangle_{H(t)}$.
\end{itemize}
Set $\theta_{\bar\beta\alpha}=\langle \sqrt{-1}\Lambda_{\omega}F_H(e_{\alpha}), e_{\beta}\rangle_H$, then we have on $U\times[0,\infty)$
\begin{gather}\frac{\partial H_{\bar\beta\alpha}}{\partial t}=-2(\theta_{\bar\beta\alpha}-\lambda H_{\bar\beta\alpha}),\\
\theta_{\bar\beta\alpha}=\sqrt{-1}\Lambda_{\omega}(\bar\partial\partial H_{\bar\beta\alpha}-\bar\partial H_{\bar\beta\gamma} H^{\gamma\bar\delta} \partial H_{\bar\delta\alpha}),
\end{gather}
where $(H^{\alpha\bar\beta})$ is the inverse matrix of $(H_{\bar\beta\alpha})$.
Because $dH_{\bar \beta \alpha}\big|_{(p_0,t_0)}=0$, it is easy to check
\begin{equation}\label{lcevlH}
(\frac{\partial}{\partial t}-2\sqrt{-1}\Lambda_{\omega}\pbp)H_{\bar\beta\alpha}\big|_{(p_0,t_0)}=2\lambda H_{\bar\beta\alpha}\big|_{(p_0,t_0)},
\end{equation}
\begin{equation}\label{lcevlK}
(\frac{\partial}{\partial t}-2\sqrt{-1}\Lambda_{\omega}\pbp)\theta_{\bar\beta\alpha}\big|_{(p_0,t_0)}=2\lambda \theta_{\bar\beta\alpha}\big|_{(p_0,t_0)}.
\end{equation}

Running the Gram-Schmidt process, we can construct smooth maps $\tilde e_1,\cdots, \tilde{e}_r:U\times[0,\infty)\rightarrow E$,
\begin{equation}
\begin{aligned}
&\tilde e_1=\frac{e_1}{|e_1|_{H}},\\
&\tilde e_2=\frac{e_2-\langle e_2,\tilde e_1\rangle_{H}\tilde e_1}{|e_2-\langle e_2,\tilde e_1\rangle_{H}\tilde e_1|_{H}},\\
&\quad \cdots,\\
&\tilde e_r=\frac{e_r-\langle e_r,\tilde e_1\rangle_{H}\tilde e_1-\cdots-\langle e_r,\tilde e_{r-1}\rangle_{H}\tilde e_{r-1}}{|e_r-\langle e_r,\tilde e_1\rangle_{H}\tilde e_1-\cdots-\langle e_r,\tilde e_{r-1}\rangle_{H}\tilde e_{r-1}|_{H}}.
\end{aligned}
\end{equation}
At any $(p,t)\in M\times[0,\infty)$, $\{\tilde e_1(p,t),\cdots,\tilde e_r(p,t)\}$ forms an orthonormal basis of $(E|_p, H(t)|_p)$. If we write
\begin{equation}\tilde e_{\alpha}=a_{\alpha}^{\beta}e_{\beta},
\end{equation}
then the coefficients $a_{\alpha}^{\beta}$ satisfy
\begin{itemize}
\item[1)] $\overline{a_{\alpha}^{\gamma}}H_{\bar\gamma\delta}a_{\beta}^{\delta}=\delta_{\alpha\beta}$; \item[2)] $a_{\alpha}^{\beta}(p_0,t_0)=\delta_{\alpha\beta}$;
\item[3)] $da_{\alpha}^{\beta}(p_0,t_0)=0$.
\end{itemize}
Define
\begin{equation}f_{L,k}=\sum_{\alpha=r-k+1}^r \langle\theta_H\tilde e_{\alpha},\tilde e_{\alpha}\rangle_H=\sum_{\alpha=r-k+1}^r\overline{a_{\alpha}^{\gamma}}\theta_{\bar\gamma\delta}a_{\alpha}^{\delta}.
\end{equation}
Obviously, $f_{L,k}(p_0,t_0)=\lambda_{L,k}(p_0,t_0)$.

According to the definition of $\lambda_{L,k}$ and the fact that $\theta_H\in \Herm(E, H)$, one has
\begin{equation}\lambda_{L,k}(p,t)=\inf \left\{\sum_{\alpha=1}^k\langle\theta_{H(t)}v_{\alpha},v_{\alpha}\rangle_{H(t)}\bigg| \{v_{\alpha}\}_{\alpha=1}^k\in E|_p \textrm{ is $H(t)$-orthonormal}\right\},
\end{equation}
consequently $f_{L,k}\geq \lambda_{L,k}$ on $U\times[0,\infty)$.

For any $1 \leq \alpha \leq r$, we  write
\begin{equation}f_{\alpha}=\langle\theta_{H}\tilde e_{\alpha},\tilde e_{\alpha}\rangle_H=\overline{a_{\alpha}^{\gamma}}\theta_{\bar\gamma\delta}a_{\alpha}^{\delta}.
\end{equation}
By the properties that $e_{\alpha}$ and $a_{\alpha}^{\beta}$ have, a direct computation yields
\begin{equation}\begin{aligned}
&(\frac{\partial}{\partial t}-2\sqrt{-1}\Lambda_{\omega}\pbp)f_{\alpha}\big|_{(p_0,t_0)}\\
=&(\frac{\partial}{\partial t}-2\sqrt{-1}\Lambda_{\omega}\pbp)\overline{a_{\alpha}^{\gamma}}\theta_{\bar\gamma\delta}a_{\alpha}^{\delta}\big|_{(p_0,t_0)}\\
=&[(\frac{\partial}{\partial t}-2\sqrt{-1}\Lambda_{\omega}\pbp)\theta_{\bar\alpha\alpha}+\theta_{\bar\alpha\alpha}(\frac{\partial}{\partial t}-2\sqrt{-1}\Lambda_{\omega}\pbp)|a_{\alpha}^{\alpha}|^2]\big|_{(p_0,t_0)}.
\end{aligned}
\end{equation}
Similarly,
\begin{equation}\begin{aligned}
0&=(\frac{\partial}{\partial t}-2\sqrt{-1}\Lambda_{\omega}\pbp)\overline{a_{\alpha}^{\gamma}}H_{\bar\gamma\delta}a_{\alpha}^{\delta}\big|_{(p_0,t_0)}\\
&=[(\frac{\partial}{\partial t}-2\sqrt{-1}\Lambda_{\omega}\pbp)H_{\bar\alpha\alpha}+(\frac{\partial}{\partial t}-2\sqrt{-1}\Lambda_{\omega}\pbp)|a_{\alpha}^{\alpha}|^2]\big|_{(p_0,t_0)}.
\end{aligned}
\end{equation}
Together with (\ref{lcevlH}) and (\ref{lcevlK}), one can deduce
\begin{equation}(\frac{\partial}{\partial t}-2\sqrt{-1}\Lambda_{\omega}\pbp)f_{\alpha}\big|_{(p_0,t_0)}=0.
\end{equation}
Since $f_{L,k}$ is a sum of several $f_{\alpha}$, we arrive at  $\left({\textstyle\frac{\partial}{\partial t}}-2\sqrt{-1}\Lambda_{\omega}\pbp\right)f_{L,k}\big|_{(p_0,t_0)}=0$.
\end{proof}

As a simple corollary of Theorem \ref{thm2'}, we have
\begin{cor}\label{lpcvgEK}
There is a constant vector $\vec\mu\in\mathbb R^r$ such that for any $p\in[1,\infty)$
\begin{equation}
\lim_{t\rightarrow\infty}||\vec\lambda(H(t),\omega)-\vec\mu||_{L^p(\omega)}=0.
\end{equation}
\end{cor}
\begin{proof} For convenience, set $\lambda_{U,0}(t)=0$ and $a_0=0$. For $1\leq k\leq r$, we define
\begin{equation}
a_k=\lim_{t\rightarrow \infty}\lambda_{mU,k}(t),\qquad b_k=a_k-a_{k-1}.
\end{equation}
We intend to show that $\vec\lambda(H(t),\omega)$ converges to $\vec\mu=(b_1,\cdots,b_r)$ in $L^p$ topology. Since $\lambda_{k}(t)=\lambda_{U,k}(t)-\lambda_{U,k-1}(t)$ and $b_k=a_k-a_{k-1}$, we only need to show that $\lambda_{U,k}(t)$ converges to $a_k$ in $L^p$ topology.

Of course $|\lambda_{U,k}(\cdot,t)|\leq C$ for some constant $C>0$. This means that $|a_k|\leq C$. By the definitions of $\hat\lambda_{U,k}(t)$ and $\lambda_{mU,k}(t)$, it holds that
\begin{equation}
||\lambda_{U,k}(\cdot,t)-a_k||_{L^1(M)}\leq (\hat\lambda_{U,k}(t)-\lambda_{mU,k}(t)+|\hat\lambda_{U,k}(t)-a_k|)\Vol(M,\omega).
\end{equation}
On account of Theorem \ref{thm2'}, $\lim_{t\rightarrow\infty}\hat\lambda_{U,k}(\cdot,t)=\lim_{t\rightarrow \infty}\lambda_{mU,k}(t)=a_k$. Then
\begin{equation}\lim_{t\rightarrow\infty}\Vert \lambda_{U,k}(\cdot,t)-a_k \Vert_{L^1(M,\omega)}=0.
\end{equation}
Using $|\lambda_{U,k}(\cdot,t)|\leq C$ again gives
\begin{equation}
||\lambda_{U,k}(\cdot,t)-a_k||_{L^p(M)}^p\leq (2C)^{p-1}||\lambda_{U,k}(\cdot,t)-a_k||_{L^1(M)},
\end{equation}
and accordingly $\lim_{t\rightarrow\infty}\Vert \lambda_{U,k}(\cdot,t)-a_k\Vert_{L^p(M)}=0$.
\end{proof}

Next we will work on the proof of Theorem \ref{thmuni}, that is,
\begin{thm}\label{estK2}  Let $H_i(t)$ ($i=1,2$) be two smooth solutions of the Hermitian-Yang-Mills flow (\ref{Flow}) on $E$ and $\theta_{H_i}(t)=\sqrt{-1}\Lambda_\omega F_{H_i(t)}$. Then we have
\begin{equation}\label{ulimK}
\lim_{t\rightarrow \infty}\int_M |\theta_{H_2}-\theta_{H_1}|^2_{H_1}(t)\frac{\omega^n}{n!}=0.
\end{equation}
\end{thm}

For convenience, set
\begin{equation}
h(t)=H_1^{-1}(t)H_2(t),\qquad A(t)=D_{H_2(t)}-D_{H_1(t)}.
\end{equation}
We also choose $\sigma(t)\in\Herm^+(E, H_1(t))$ to be the unique element satisfying
\begin{equation}
\sigma^{*_{H_1}}(t)\sigma(t)=\sigma^2(t)=h(t).
\end{equation}

To prove Theorem \ref{estK2}, we need the following lemma.
\begin{lem}\label{estA}
Let $C=\sup_M \tr(h+h^{-1})(0)-2r$. Then
\begin{gather}
\label{bdh}
\sup_M \tr(h+h^{-1})(t)-2r\leq C,\\
\label{intA}
\int_0^\infty\int_M|A|^2_{H_1,\omega}(t)\frac{\omega^n}{n!}dt\leq \frac{C(2+C)}{4}\Vol(M,\omega).
\end{gather}
Furthermore, we have
\begin{equation}\label{limA}
\lim_{t\rightarrow \infty}\int_M|A|^2_{H_1,\omega}(t)\frac{\omega^n}{n!}=0.
\end{equation}
\end{lem}
\begin{proof}
The evolution equation of $h(t)$ is
\begin{equation}
(\frac{\partial}{\partial t}+2\sqrt{-1}\Lambda_\omega\bar\partial\partial_{H_1})h(t)=2([\theta_{H_1},h]+\sqrt{-1}\Lambda_\omega (\bar{\partial }h\wedge h^{-1}\partial_{H_1}h)),
\end{equation}
consequently we have
\begin{equation}\begin{aligned}
&\quad(\frac{\partial}{\partial t}-2\sqrt{-1}\Lambda_\omega \pbp)\tr(h +h^{-1})\\
&=2 \sqrt{-1}\Lambda_\omega\tr(\bar\partial h\wedge h^{-1}\partial_{H_1}h-h^{-1}\partial_{H_1}h\wedge h^{-1}\bar\partial hh^{-1} ).
\end{aligned}
\end{equation}
Notice that $A=h^{-1}\partial_{H_1}h$, $\bar\partial h h^{-1}=A^{*_{H_1}}$ and $\sigma^2=h$. One can find that
\begin{equation}\begin{aligned}
-\sqrt{-1}\Lambda_{\omega}\tr(\bar\partial h\wedge h^{-1}\partial_{H_1}h)&=-\sqrt{-1}\Lambda_\omega\tr((\sigma A)^{*_{H_1}}\wedge(\sigma A))\\
&=|\sigma A|^2_{H_1,\omega}
\end{aligned}\end{equation}
and
\begin{equation}\begin{aligned}
&\sqrt{-1}\Lambda_{\omega}\tr(h^{-1}\partial_{H_1}h\wedge h^{-1}\bar\partial hh^{-1})\\= & \sqrt{-1}\Lambda_\omega\tr((A\sigma^{-1})\wedge(A\sigma^{-1})^{*_{H_1}})\\
= & |A\sigma^{-1}|^2_{H_1,\omega}.
\end{aligned}\end{equation}
Thus
\begin{equation}\label{evlh'}
(\frac{\partial}{\partial t}-2\sqrt{-1}\Lambda_\omega \pbp)(\tr(h+ h^{-1}) -2r)= -2\left(|\sigma A|^2_{H_1,\omega}+|A\sigma^{-1}|^2_{H_1,\omega}\right).
\end{equation}
The parabolic maximal principle tells us that
\begin{equation}
\sup_M\tr(h+h^{-1})-2r\leq \sup_M \tr(h+h^{-1})(0)-2r=C.
\end{equation}
This means
\begin{equation}
\frac{1}{2+C}H_1\leq H_2\leq (2+C)H_1,
\end{equation}
and
\begin{equation}
\frac{1}{2+C}\Id_E\leq h\leq (2+C)\Id_E,\qquad \frac{1}{\sqrt{2+C}}\Id_E\leq \sigma\leq \sqrt{2+C}\Id_E
\end{equation}
in $\Herm^+(E, H_1)$. Then
\begin{equation}
(\frac{\partial}{\partial t}-2\sqrt{-1}\Lambda_\omega \pbp)\left(\tr(h+ h^{-1}) -2r\right)\leq -\frac{4}{2+C}|A|^2_{H_1,\omega},
\end{equation}
which implies (\ref{intA}),
\begin{equation}\begin{aligned}
\int_0^\infty\int_M|A|^2_{H_1,\omega}(t)\frac{\omega^n}{n!}dt &\leq \frac{2+C}{4}\int_M\left(\tr(h+ h^{-1}) -2r\right)(0)\frac{\omega^n}{n!}\\
&\leq \frac{C(2+C)}{4}\Vol(M,\omega).
\end{aligned}\end{equation}

Write $C_1=2+C+\sup_M\left(|\theta_{H_1}|_{H_1}+|\theta_{H_2}|_{H_2}\right)(0)$, then we always have
\begin{equation}
2+\tr(h+h^{-1})(t)-2r+|\theta_{H_1}|_{H_1}(t)+|\theta_{H_2}|_{H_2}(t)\leq C_1.
\end{equation}
The evolution equation of $A(t)$ is
\begin{equation}
\frac{\partial A(t)}{\partial t}=-2\partial_{H_2}\theta_{H_2}(t)+2\partial_{H_1}\theta_{H_1}(t),
\end{equation}
so
\begin{equation}\begin{aligned}
\left|\frac{\partial A(t)}{\partial t}\right|_{H_1,\omega} &\leq 2|\partial_{H_2}\theta_{H_2}|_{H_1,\omega}(t)+2|\partial_{H_1}\theta_{H_1}|_{H_1,\omega}(t)\\
&\leq 2C_1|\partial_{H_2}\theta_{H_2}|_{H_2,\omega}(t)+2|\partial_{H_1}\theta_{H_1}|_{H_1,\omega}(t).
\end{aligned}\end{equation}
This yields that
\begin{equation}\begin{aligned}
\frac{\partial}{\partial t}|A|^2_{H_1,\omega}(t)&=2\langle[A,\theta_{H_1}], A\rangle_{H_1,\omega}+2\real\Big\langle \frac{\partial A}{\partial t},A \Big\rangle_{H_1,\omega}\\
&\leq 4C_1|A|^2_{H_1,\omega}+4C_1|A|_{H_1,\omega}|\partial_{H_2}\theta_{H_2}|_{H_2,\omega}\\
&\quad+4|A|_{H_1,\omega}|\partial_{H_1}\theta_{H_1}|_{H_1,\omega}\\
&\leq 8C_1|A|^2_{H_1,\omega}+2(C_1|\partial_{H_2}\theta_{H_2}|^2_{H_2,\omega}+|\partial_{H_1}\theta_{H_1}|^2_{H_1,\omega}).
\end{aligned}\end{equation}
Consider
\begin{equation}
f(t)=\int_M |A|^2_{H_1,\omega}(t)\frac{\omega^n}{n!}
\end{equation}
and
\begin{equation}
a(t)=2\int_M(C_1|\partial_{H_2}\theta_{H_2}|^2_{H_2,\omega}+|\partial_{H_1}\theta_{H_1}|^2_{H_1,\omega})(t)\frac{\omega^n}{n!}.
\end{equation}
There holds that
\begin{equation}\label{ulcw}
\lim_{t\rightarrow\infty}\int_{t-1}^{t}(f(s)+a(s))\dif s=0
\end{equation}
and
\begin{equation}
f'(t)\leq 8C_1f(t)+a(t).
\end{equation}
For any fixed $t_0\in[0,\infty)$,  we have
\begin{equation}
(e^{-8C_1(t-t_0)}f)'(t)\leq a(t),  \qquad \text{for} \ t\geq t_0.
\end{equation}
Then on $[t_0, +\infty)$,
\begin{equation}
f(t)\leq e^{8C_1(t-t_0)}\left(f(t_0)+\int_{t_0}^{t}a(s)\dif s\right).
\end{equation}
Now assume that $t\in[1,\infty)$. It follows that for any $s\in[t-1,t]$,
\begin{equation}
f(t)\leq e^{8C_1}\left(f(s)+\int_{t-1}^ta(x)\dif x\right).
\end{equation}
Integrating the right hand side with respect to $s$ on $[t-1,t]$, one can get
\begin{equation}
f(t)\leq e^{8C_1}\int_{t-1}^t(f(s)+a(s))\dif s.
\end{equation}
Combined with (\ref{ulcw}), this gives the desired equality (\ref{limA}).
\end{proof}

\begin{proof}[Proof of Theorem \ref{estK2}]
Obviously it holds that
\begin{equation}
\theta_{H_2}-\theta_{H_1}=\sqrt{-1}\Lambda_{\omega}\bar\partial A=[\sqrt{-1}\Lambda_{\omega},\bar\partial]A=(\partial^{*_{H_1}}+\tau^*)A.
\end{equation}
A simple calculation implies
\begin{equation}\begin{aligned}
&\quad\int_M|\theta_{H_2}-\theta_{H_1}|^2_{H_1}\frac{\omega^n}{n!}\\
&=\int_M\langle (\partial^{*_{H_1}}+\tau^*)A,\theta_{H_2}-\theta_{H_1}\rangle_{H_1}\frac{\omega^n}{n!}\\
&=\int_M\langle A,(\partial_{H_1}+\tau)(\theta_{H_2}-\theta_{H_1})\rangle_{H_1,\omega}\frac{\omega^n}{n!}\\
&=\int_M\langle A,\tau (\theta_{H_2}-\theta_{H_1})\rangle_{H_1,\omega}\frac{\omega^n}{n!}
+\int_M\langle A,\partial_{H_2}\theta_{H_2}\rangle_{H_1,\omega}\frac{\omega^n}{n!}\\
&\quad-\int_M\langle A,[A,\theta_{H_2}]\rangle_{H_1,\omega}\frac{\omega^n}{n!}
-\int_M\langle A,\partial_{H_1}\theta_{H_1}\rangle_{H_1,\omega}\frac{\omega^n}{n!}.
\end{aligned}
\end{equation}
Clearly there is a constant $C$ depending only on $\omega$, $H_1(0)$ and $H_2(0)$ such that
\begin{equation}
\tr(h+h^{-1})+ |\tau|_{\omega}+|\theta_{H_1}|_{H_1}+|\theta_{H_2}|_{H_2}\leq C.
\end{equation}
Then we can find a constant $C_1$ depending only on $C$, such that
\begin{equation}\begin{aligned}
&\quad\int_M|\theta_{H_2}-\theta_{H_1}|^2_{H_1}\frac{\omega^n}{n!}\\
&\leq C_1\int_M|A|_{H_1,\omega}\left(1+|A|_{H_1,\omega}+|\partial_{H_1}\theta_{H_1}|_{H_1,\omega}+|\partial_{H_2}\theta_{H_2}|_{H_2,\omega}\right)\frac{\omega^n}{n!}.
\end{aligned}
\end{equation}
On account of (\ref{limDK}) and (\ref{limA}), the right hand term converges to $0$ when $t\rightarrow\infty$, which concludes the proof.
\end{proof}

Point-wisely by choosing an orthonormal basis with respect to $H_1$, we can treat $\sigma$, $\theta_{H_2}$ and $\theta_{H_1}$ as matrices. Since $\sigma\theta_{H_2}\sigma^{-1}$ and $\theta_{H_1}$ are Hermitian, one can find unitary matrices $U_1$ and $U_2$, such that $U_1\sigma\theta_{H_2}\sigma^{-1}U_1^{-1}$ and $U_2\theta_{H_1}U_2^{-1}$ are real diagonal matrices. According to \cite[Theorem VIII.3.9]{Bha97}, we know
\begin{equation*}
|\vec\lambda(H_2)-\vec\lambda(H_1)|^2\leq \cond(U_1\sigma)\cond(U_2)|\theta_{H_2}-\theta_{H_1}|^2_{H_1},
\end{equation*}
where  $\cond(B)$ is the condition number of a nonsingular matrix $B$.
Because $U_1$ and $U_2$ are unitary, one has $\cond(U_1\sigma)=\cond(\sigma)$ and $\cond(U_2)=1$. Thus
\begin{equation}\label{estEK}
|\vec\lambda(H_2)-\vec\lambda(H_1)|^2\leq \cond(\sigma)|\theta_{H_2}-\theta_{H_1}|^2_{H_1}.
\end{equation}
Moreover, $\cond (\sigma)$ is equal to the quotient of the largest eigenvalue by the  smallest eigenvalue of $\sigma$. We observe
\begin{equation}
\cond(\sigma)\leq \frac{1}{2}\tr(h+h^{-1})\leq \frac{1}{2}\sup_M\tr(h+h^{-1})(0).
\end{equation}

From Theorem \ref{estK2}, we immediately have
\begin{lem}\label{lem3'}
Let $H_1(t)$ and $H_2(t)$ be two smooth solutions of the Hermitian-Yang-Mills flow (\ref{Flow}). Then
\begin{equation}
\label{lm0}
\lim_{t\rightarrow\infty}\int_M\big|\vec\lambda(H_1(t),\omega)-\vec\lambda(H_2(t),\omega)\big|^2\frac{\omega^n}{n!}=0.
\end{equation}
\end{lem}

Now we are able to prove Theorem \ref{thm3}. According to  Corollary \ref{lpcvgEK}, it is enough to show that
\begin{thm}\label{thm3'}
Let $H(t)$ be a smooth solution of the Hermitian-Yang-Mills flow (\ref{Flow}). Then for $1 \leq k \leq r$,  we have
\begin{align}
\label{limmL0}
&\lim_{t\rightarrow\infty}\lambda_{mL,k}(H(t),\omega)= \frac{2\pi}{\Vol(M,\omega)}\sum_{i=1}^k\mu_{r+1-i,\omega}(E),\\
\label{limmU0}
&\lim_{t\rightarrow\infty}\lambda_{mU,k}(H(t),\omega)= \frac{2\pi}{\Vol(M,\omega)}\sum_{i=1}^k\mu_{i,\omega}(E).
\end{align}
\end{thm}

Once the following lemma is proved, the ``$\leq$" part of (\ref{limmL0}) and ``$\geq$" part of (\ref{limmU0}) will  come.
\begin{lem}\label{eigest1}
Let $H$ be a smooth Hermitian metric on $E$. Then for $1 \leq k \leq r$, we have
\begin{align}
\label{mL1}
&\lambda_{mL,k}(H,\omega)\leq \frac{2\pi}{\Vol(M,\omega)}\sum_{i=1}^k\mu_{r+1-i,\omega}(E),\\
\label{mU1}
&\lambda_{mU,k}(H,\omega)\geq \frac{2\pi}{\Vol(M,\omega)}\sum_{i=1}^k\mu_{i,\omega}(E).
\end{align}
\end{lem}

\begin{proof}
We only give the proof of (\ref{mU1}), the proof of (\ref{mL1}) is similar.

Let $0=\mathcal{E}_0\subset \mathcal{E}_1\subset\cdots\subset \mathcal{E}_l=E$ be the Harder-Narasimhan filtration of  $E$, and set
\begin{equation}
r_j=\rank(\mathcal{E}_j)
\end{equation}
for $1 \leq j \leq l$.
Recall the following Chern-Weil formula (\cite{Bru}):
\begin{equation}\label{B1}
\begin{split}
\deg_{\omega }(\mathcal{E}_j)=&\int_{M\setminus \Sigma_{alg}}\frac{\sqrt{-1}}{2\pi}\tr F_{H_{\mathcal{E}_j}}\wedge \frac{\omega^{n-1}}{(n-1)!}\\
=& \frac{1}{2\pi}\int_{M\setminus \Sigma_{alg}} (\sqrt{-1}\textmd{tr}(\pi_{\mathcal{E}_j}^{H} \Lambda_{\omega} F_{H})-|\bar{\partial }\pi_{\mathcal{E}_j}^{H} |_{H}^{2})\frac{\omega^{n}}{n!},
\end{split}
\end{equation}
where $\Sigma_{alg}$ is the singular set of $\mathcal{E}_j$, $H_{\mathcal{E}_j}$ is the induced metric on $\mathcal{E}_j|_{M\setminus \Sigma_{alg}}$ and $\pi_{\mathcal{E}_j}^{H}$ is the orthogonal projection onto $\mathcal{E}_j$ with respect to the metric $H$.
Together with the definition of $\vec\mu_\omega (E)$, one can easily verify that
\begin{equation}\label{mUeq1}
\lambda_{mU,r_j}(H,\omega)\geq \frac{2\pi}{\Vol(M,\omega)}\deg_{\omega}(\mathcal{E}_j)=\frac{2\pi}{\Vol(M,\omega)}\sum_{i=1}^{r_j}\mu_{i,\omega}(E).
\end{equation}
We will proceed by induction on $k$.

\noindent 1) For $k=1$, we have
\begin{equation}
r_1\lambda_{mU,1}\geq\lambda_{mU,r_1}(H,\omega)\geq \frac{2\pi}{\Vol(M,\omega)}\sum_{i=1}^{r_1}\mu_{i,\omega}(E)= \frac{2r_1\pi}{\Vol(M,\omega)}\mu_{1,\omega}(E).
\end{equation}
So (\ref{mU1}) holds for $k=1$.

\noindent 2) Assume that $1\leq s\leq r-1$ and (\ref{mU1}) holds for $k=s$. Namely
\begin{equation}\label{mUeq2}
\lambda_{mU,s}(H,\omega)\geq \frac{2\pi}{\Vol(M,\omega)}\sum_{i=1}^s\mu_{i,\omega}(E).
\end{equation}
If $s+1=r_j$ for some $j=1,\cdots,l$, then of course (\ref{mU1}) holds for $k=s+1$. Otherwise we can find some $j=1,\cdots,l$ such that $r_{j-1}<s+1<r_j$. Based on (\ref{mUeq1}),  one can see
\begin{equation}\begin{aligned}
&\quad\lambda_{mU,s}(H,\omega)+\frac{r_j-s}{\Vol(M,\omega)}\int_M\lambda_{s+1}(H,\omega)\frac{\omega^n}{n!}\\
&\geq \lambda_{mU,r_j}(H,\omega)\\
&\geq \frac{2\pi}{\Vol(M,\omega)}\sum_{i=1}^{r_j}\mu_{i,\omega}(E)\\
&=\frac{2\pi}{\Vol(M,\omega)}\left(\sum_{i=1}^{s}\mu_{i,\omega}(E)+(r_j-s)\mu_{s+1,\omega}(E)\right).
\end{aligned}\end{equation}
Making use of (\ref{mUeq2}), we deduce
\begin{equation}
(r_j-s)\lambda_{mU,s+1}(H,\omega)\geq \frac{2(r_j-s)\pi}{\Vol(M,\omega)}\sum_{i=1}^{s+1}\mu_{i,\omega}(E).
\end{equation}
Hence (\ref{mU1}) holds for $k=s+1$.
This concludes the proof.
\end{proof}

In order to show the ``$\geq$" part of (\ref{limmL0}) and the ``$\leq$" part of (\ref{limmU0}), we only need to apply Lemma \ref{lem3'} and the following  lemma.

\begin{lem}\label{lem1'}
For any $\delta>0$, we can find a smooth solution  $H_\delta(t)$ of the Hermitian-Yang-Mills flow (\ref{Flow}) on $(E, \bar{\partial }_{E})$ such that for any $1 \leq k \leq  r$,
\begin{align}
\label{limmL1}
&\lim_{t\rightarrow\infty}\lambda_{mL,k}(H_\delta(t),\omega)\geq \frac{2\pi}{\Vol(M,\omega)}\sum_{i=1}^k\mu_{r+1-i,\omega}(E)-k\delta,\\
\label{limmU1}
&\lim_{t\rightarrow\infty}\lambda_{mU,k}(H_\delta(t),\omega)\leq \frac{2\pi}{\Vol(M,\omega)}\sum_{i=1}^k\mu_{i,\omega}(E)+k\delta.
\end{align}
\end{lem}
\begin{proof}Apply Theorem \ref{theorem1} and Theorem \ref{thm2'}.
\end{proof}

\section{Some applications}

\subsection{Application to the calculation of the Harder-Narasimhan type}
Using Theorem \ref{thm3}, we can easily calculate the Harder-Narasimhan types of tensor products, symmetric and exterior powers of holomorphic vector bundles.

Firstly we review the symmetric and exterior powers of a complex or real vector space.
Let $r\geq 2$, $k\geq 1$ and  assume $V$ is an $r$-dimensional complex (resp.  real) vector space. We choose a basis $\{e_1,\cdots,e_r\}$ of $V$ and write
\begin{equation}
{\rm T\Lambda}(k,r)=\{(i_1,\cdots,i_k) \; | \; i_l=1,\cdots,r, \; \textrm{for every}\; 1 \leq  l \leq k\}.
\end{equation}
Then a basis of the $k$th tensor power $V^{\otimes k}$ can be chosen as
\begin{equation}
\{e_{i_1}\otimes\cdots\otimes e_{i_k} \; | \; (i_1,\cdots,i_k)\in {\rm T\Lambda}(k,r)\}.
\end{equation}
The symmetrization ${\rm Sym}:V^{\otimes k}\rightarrow V^{\otimes k}$ is the linear operator  satisfying
\begin{equation}
{\rm Sym} (e_{i_1}\otimes\cdots\otimes e_{i_k})=\frac{1}{k!}\sum_{\sigma} e_{i_{\sigma(1)}}\otimes\cdots\otimes e_{i_{\sigma(k)}}
\end{equation}
for any $(i_1,\cdots,i_k)\in {\rm T\Lambda}(k,r)$, where the summation is taken over all permutations $\sigma$ of $\{1,\cdots,k\}$.
The $k$th symmetric power $S^k V$ is the set of all elements  invariant under the symmetrization ${\rm Sym}$.
Clearly $({\rm Sym})^2={\rm Sym}$, so $S^k V={\rm Sym}(V^{\otimes k})$ and we have the natural decomposition
\begin{equation}
V^{\otimes k}\xrightarrow{({\rm Sym}, \, \Id-{\rm Sym})}S^k V\oplus \ker{\rm Sym}.
\end{equation}
Furthermore, set
\begin{equation}
{\rm S\Lambda}(k,r)=\{(a_1,\cdots,a_r)\in\mathbb Z_{\geq 0} \; | \; a_1+\cdots+a_r=k\}
\end{equation}
and for every $a=(a_1,\cdots,a_r)\in {\rm S\Lambda}(k,r)$,
\begin{equation}
e^a=\frac{k!}{a!}{\rm Sym}(e_1^{\otimes a_1}\otimes \cdots\otimes e_r^{\otimes a_r}),
\end{equation}
where $a!=a_1!\cdots a_r!$.
Then a basis of $S^k V$ can be chosen as
\begin{equation}
\{e^a \; | \; a\in {\rm S\Lambda}(k,r)\}.
\end{equation}
The anti-symmetrization (or skew-symmetrization) ${\rm Alt}:V^{\otimes k}\rightarrow V^{\otimes k}$ is the linear operator  satisfying
\begin{equation}
{\rm Alt} (e_{i_1}\otimes\cdots\otimes e_{i_k})=\frac{1}{k!}\sum_{\sigma} {\rm sgn}(\sigma) e_{i_{\sigma(1)}}\otimes\cdots\otimes e_{i_{\sigma(k)}}
\end{equation}
for any $(i_1,\cdots,i_k)\in {\rm T\Lambda}(k,r)$, where the summation is taken over all permutations $\sigma$ of $\{1,\cdots,k\}$ and  ${\rm sgn}(\sigma)$ is the sign of $\sigma$.
Actually ${\rm Alt}=0$ if $k>r$.
For $k\leq r$, the $k$th exterior (or alternating)  power $\wedge^k V$ is the set of all elements  invariant under the anti-symmetrization ${\rm Alt}$.
Of course $({\rm Alt})^2={\rm Alt}$, hence $\wedge^k V={\rm Alt}(V^{\otimes k})$ and we have the natural decomposition
\begin{equation}
V^{\otimes k}\xrightarrow{({\rm Alt}, \, \Id-{\rm Alt})}\wedge^k V\oplus \ker{\rm Alt}.
\end{equation}
Usually we write
\begin{equation}
v_1\wedge\cdots\wedge v_k=k!{\rm Alt}(v_1\otimes\cdots\otimes v_k).
\end{equation}
Let
\begin{equation}
{\rm A\Lambda}(k,r)=\{(i_1,\cdots,i_k)\in {\rm T\Lambda}(k,r) \; | \; i_1<\cdots<i_k\}.
\end{equation}
Then a basis of $\wedge^k V$ can be chosen as
\begin{equation}
\{e_{i_1}\wedge\cdots\wedge e_{i_k} \; | \; (i_1,\cdots,i_k)\in {\rm A\Lambda}(k,r)\}.
\end{equation}

Fix an inner product on $V$ and suppose $\tilde V$ is an $s$-dimensional  ($s\geq 1$) complex (resp. real) vector space endowed with an inner product. Then there is a natural inner product on $V\otimes\tilde V$, uniquely determined by
\begin{equation}
\langle v_1\otimes\tilde v_1,v_2\otimes \tilde v_2\rangle=\langle v_1,v_2\rangle\langle\tilde v_1,\tilde v_2\rangle
\end{equation}
for $v_1, v_2 \in V$ and $\tilde v_1,  \tilde v_2\in  \tilde V$.
Besides, let $\{e_1,\cdots,e_r\}$ be an orthonormal basis of $V$ and $\{\tilde e_1,\cdots,\tilde e_s\}$ an orthonormal basis of $\tilde V$, then
\begin{equation}
\{e_i\otimes \tilde e_j \; | \; 1\leq i\leq r,1\leq j\leq s\}
\end{equation}
is an orthonormal basis of $V\otimes \tilde V$.

Notice that the operators ${\rm Sym}$ and ${\rm Alt}$ are self-adjoint with respect to the induced inner product on $V^{\otimes k}$. Thus the decompositions $V^{\otimes k}=S^k V\oplus \ker {\rm Sym}$ and $V^{\otimes k}=\wedge^k V\oplus \ker {\rm Alt}$ (when $k\leq r$) are actually orthogonal. Furthermore,
\begin{equation}
\left\{{\textstyle\sqrt{ \frac{a!}{k!}}}e^a \; \bigg| \; a\in {\rm S\Lambda}(k,r)\right\}
\end{equation}
 is an orthonormal basis of $S^k V$ and when $k\leq r$,
\begin{equation}
\left\{{\textstyle \frac{1}{\sqrt{k!}}}e_{i_1}\wedge\cdots\wedge e_{i_k} \; \bigg| \; (i_1,\cdots,i_k)\in {\rm A\Lambda}(k,r)\right\}
\end{equation}
is an orthonormal basis of $\wedge^k V$.

The notions and decompositions above can be naturally generalized to the case of complex vector bundles.

When expressing the eigenvalues of self-adjoint operators on an $r$-dimensional complex or real vector space, it is convenient to use the space
\begin{equation}
\mathbb R^r_{\downarrow}:=\{(x_1,\cdots,x_r)\in\mathbb R^r \; | \; x_1\geq\cdots\geq x_r\}.
\end{equation}
For convenience, we introduce the map $\tau:\mathbb R^r\rightarrow \mathbb R^r_{\downarrow}$ defined as follows. For any $x=(x_1,\cdots,x_r)\in\mathbb R^r$, define $\tau(x)=(y_1,\cdots,y_r)$ to be  the unique vector such that
\begin{equation}
(y_1,\cdots,y_r)=(x_{\sigma(1)},\cdots,x_{\sigma(r)}) \in\mathbb R^r_{\downarrow}
\end{equation}
for some permutation $\sigma$. Obviously $\tau$ is well-defined. Moreover, the well-known rearrangement inequality implies
\begin{equation}
|\tau(x)-\tau(\tilde x)|\leq |x-\tilde x|,\qquad \forall \,  x,\tilde x\in\mathbb R^r.
\end{equation}

For later use, we give the following simple fact without proof.
\begin{prop}\label{limtcomm}
Let $D\subset \mathbb R^r$  be a closed subset, $T:D\rightarrow \mathbb R^N(N\geq 1)$ a Lipschitz map and $(X,g)$ a Riemannian manifold. If  $\{f_i\}$ is a sequence of measurable maps from $X$ into $D$ and
\begin{equation}
\lim_{i\rightarrow\infty}||f_i-v||_{L^1(X)}=0
\end{equation}
for some constant vector $v\in D$, then
\begin{equation}
\lim_{i\rightarrow\infty}||T(f_i)-T(v)||_{L^1(X)}=0.
\end{equation}
\end{prop}

Let $(M,\omega)$ be a compact Gauduchon manifold of dimension $n\geq 2$, $E$ a holomorphic vector bundle of rank $r\geq 2$ over $M$ and $\tilde E$ a holomorphic vector bundle  of rank $s\geq 1$ over $M$. Suppose $H$ and $\tilde H$  are the Hermitian metrics on $E$ and $\tilde E$, respectively. We fix an arbitrary point $z\in M$ and choose the orthogonal bases $\{e_1,\cdots,e_r\}$ of $E|_z$ and $\{\tilde e_1,\cdots,\tilde e_s\}$ of $\tilde E|_z$, respectively, such that
\begin{equation}
\sqrt{-1}\Lambda_\omega F_H|_z e_i=\lambda_i(H,\omega)(z)e_i,\quad\sqrt{-1}\Lambda_\omega F_{\tilde H}|_z \tilde e_j=\lambda_j(\tilde H,\omega)(z) \tilde e_j.
\end{equation}
Let $k$ be a positive integer. Next we discuss the mean curvatures of the tensor products, symmetric and exterior powers of the holomorphic vector bundles.
\begin{enumerate}
\item Denote by $H\otimes \tilde H$ the induced Hermitian metric on $E\otimes \tilde E$. We have
\begin{equation}
F_{H\otimes \tilde H}=F_H\otimes\Id_{\tilde E}+\Id_E\otimes F_{\tilde H},
\end{equation}
and consequently
\begin{equation}
\sqrt{-1}\Lambda_\omega F_{H\otimes \tilde H}=(\sqrt{-1}\Lambda_\omega F_H)\otimes \Id_{\tilde E}+\Id_E\otimes (\sqrt{-1}\Lambda_\omega F_{\tilde H}).
\end{equation}
Hence there holds
\begin{equation}
\sqrt{-1}\Lambda_\omega F_{H\otimes \tilde H}|_z (e_i\otimes \tilde e_j)=\big(\lambda_i(H,\omega)(z)+\lambda_j(\tilde H,\omega)(z)\big)e_i\otimes\tilde e_j
\end{equation}
for every $1\leq i\leq r$ and $1\leq j\leq s$.

\item Denote the induced Hermitian metric on $E^{\otimes k}$ by $H^{\otimes k}$. One can find
\begin{equation}
F_{H^{\otimes k}}=\sum_{i=1}^k \Id_{E^{\otimes (i-1)}}\otimes F_H\otimes\Id_{E^{\otimes (k-i)}},
\end{equation}
and consequently
\begin{equation}
\sqrt{-1}\Lambda_\omega F_{H^{\otimes k}}=\sum_{i=1}^k \Id_{E^{\otimes (i-1)}}\otimes (\sqrt{-1}\Lambda_\omega F_H)\otimes\Id_{E^{\otimes (k-i)}}.
\end{equation}
Then it follows that
\begin{equation}
\sqrt{-1}\Lambda_\omega F_{H^{\otimes k}}|_z (e_{i_1}\otimes\cdots\otimes e_{i_k})=\bigg(\sum_{l=1}^k \lambda_{i_l}(H,\omega)(z)\bigg)e_{i_1}\otimes\cdots\otimes e_{i_k}
\end{equation}
for every $(i_1,\cdots,i_k)\in {\rm T\Lambda}(k,r)$.

\item Set $S^k H=H^{\otimes k}|_{S^k E}$.  Since the decomposition
\begin{equation}
E^{\otimes k}=S^kE\oplus \ker {\rm Sym}
\end{equation}
is holomorphic and orthogonal with respect to $H^{\otimes k}$, we have
\begin{equation}
F_{S^k H}=F_{H^{\otimes k}}|_{S^k E},\qquad \sqrt{-1}\Lambda_\omega F_{S^k H}=(\sqrt{-1}\Lambda_\omega F_{H^{\otimes k}})|_{S^k E}.
\end{equation}
Thus it is easy to verify that
\begin{equation}
\sqrt{-1}\Lambda_\omega F_{S^kH}|_z (e^a)=\bigg(\sum_{i=1}^r a_i\lambda_{i}(H,\omega)(z)\bigg)e^a
\end{equation}
for every $a=(a_1,\cdots,a_r)\in {\rm S\Lambda}(k,r)$.

\item For $k\leq r$, set $\wedge^k H=H^{\otimes k}|_{\wedge^k E}$. Because the decomposition
\begin{equation}
E^{\otimes k}=\wedge^kE\oplus \ker {\rm Alt}
\end{equation}
is holomorphic and orthogonal with respect to $H^{\otimes k}$, we know
\begin{equation}
F_{\wedge^k H}=F_{H^{\otimes k}}|_{\wedge^k E},\qquad F_{\wedge^k H}=(\sqrt{-1}\Lambda_\omega F_{H^{\otimes k}})|_{\wedge^k E}.
\end{equation}
So one can see
\begin{equation}
\sqrt{-1}\Lambda_\omega F_{\wedge^kH}|_z (e_{i_1}\wedge\cdots\wedge e_{i_k})=\bigg(\sum_{l=1}^k \lambda_{i_l}(H,\omega)(z)\bigg)e_{i_1}\wedge\cdots\wedge e_{i_k}
\end{equation}
for every $(i_1,\cdots,i_k)\in {\rm A\Lambda}(k,r)$.
\end{enumerate}

By the arguments above,  all the eigenvalues of $\sqrt{-1}\Lambda_\omega F_{H\otimes \tilde H}$, $\sqrt{-1}\Lambda_\omega F_{H^{\otimes k}}$, $\sqrt{-1}\Lambda_\omega F_{S^kH}$ and $\sqrt{-1}\Lambda_\omega F_{\wedge^kH}$ are computed. After sorting the eigenvalues in the descending order, we can eventually determine $\vec\lambda(H\otimes\tilde H,\omega)$, $\vec\lambda(H^{\otimes k},\omega)$, $\vec\lambda(S^kH ,\omega)$ and $\vec\lambda(\wedge^kH,\omega)$. To give the specific expressions, we introduce the following notations.

\begin{enumerate}
\item Let $N=rs=\dim (\mathbb R^r\otimes\mathbb R^s)$ and
\begin{equation}
\Lambda(r,s)=\{(i,j)\in\mathbb Z^2 \; | \; 1\leq i\leq r,1\leq j\leq s\}.
\end{equation}
Define $\vec T:\mathbb R^r\times \mathbb R^s\rightarrow \mathbb R^{N}_{\downarrow}$ to be
\begin{equation}
\vec T((x_1,\cdots,x_r),(y_1,\cdots,y_s))=\tau(t(\varrho(1)),\cdots,t(\varrho(N))),
\end{equation}
where $\varrho:\{1,\cdots,N\}\rightarrow \Lambda(r,s)$ is an arbitrary bijective map and
\begin{equation}
t(i,j)=x_i+y_j,\qquad \forall \, (i,j)\in \Lambda(r,s).
\end{equation}

\item Let $N=r^k=\dim ((\mathbb R^r)^{\otimes k})$.
Define $\vec T_k:\mathbb R^r\rightarrow\mathbb R^{N}_{\downarrow}$  to be
\begin{equation}
\vec T_k(x_1,\cdots,x_r)=\tau(t(\varrho(1)),\cdots,t(\varrho(N))),
\end{equation}
where $\varrho:\{1,\cdots,N\}\rightarrow {\rm T\Lambda}(k,r)$ is an arbitrary bijective map and
\begin{equation}
t(i_1,\cdots,i_k)=\sum_{l=1}^kx_{i_l},\qquad \forall \,  (i_1,\cdots,i_k)\in {\rm T\Lambda}(k,r).
\end{equation}

\item  Let $N=\binom{r+k-1}{k}=\dim (S^k(\mathbb R^r))$. Define $\vec S_k:\mathbb R^r\rightarrow\mathbb R^{N}_{\downarrow}$ to be
\begin{equation}
\vec S_k(x_1,\cdots,x_r)=\tau(t(\varrho(1)),\cdots,t(\varrho(N))),
\end{equation}
where $\varrho:\{1,\cdots,N\}\rightarrow {\rm S\Lambda}(k,r)$ is an arbitrary bijective map and
\begin{equation}
t(a_1,\cdots,a_r)=\sum_{i=1}^ra_ix_{i},\qquad \forall \,  (a_1,\cdots,a_r)\in {\rm S\Lambda}(k,r).
\end{equation}

\item Assume that $k\leq r$ and let $N=\binom{r}{k}=\dim (\wedge^k(\mathbb R^r))$.  Define $\vec A_k:\mathbb R^r\rightarrow\mathbb R^N_{\downarrow}$ to be
\begin{equation}
\vec A_k(x_1,\cdots,x_r)=\tau(t(\varrho(1)),\cdots,t(\varrho(N))),
\end{equation}
where $\varrho:\{1,\cdots,N\}\rightarrow {\rm A\Lambda}(k,r)$ is an arbitrary bijective map and
\begin{equation}
t (i_1,\cdots,i_k)=\sum_{l=1}^kx_{i_l},\qquad \forall  \, (i_1,\cdots,i_k)\in {\rm A\Lambda}(k,r).
\end{equation}
\end{enumerate}
Note that every map constructed  above is globally Lipschitz and doesn't depend on the choice of $\varrho$.

Now the following formulas can be summarised:
\begin{enumerate}
\item $\vec\lambda(H\otimes \tilde H,\omega)=\vec T(\vec\lambda(H,\omega),\vec\lambda(\tilde H,\omega))$;
\item $\vec\lambda(H^{\otimes k},\omega)=\vec T_k(\vec\lambda(H,\omega))$;
\item $\vec\lambda(S^kH,\omega)=\vec S_k(\vec\lambda(H,\omega))$;
\item $\vec\lambda(\wedge^k H,\omega)=\vec A_k(\vec\lambda(H,\omega))$ when $k\leq r$.
\end{enumerate}

Furthermore, suppose $H(t)$ and $\tilde H(t)$  are the solutions of the Hermitian-Yang-Mills flows on $E$ and  $\tilde E$, respectively. It is not hard to find  that $H(t)\otimes \tilde H(t)$, $H(t)^{\otimes k}$, $S^kH(t)$ and $\wedge ^k H(t)$ (when $k\leq r$) are the solutions of the Hermitian-Yang-Mills flows on $E\otimes \tilde E$, $E^{\otimes k}$, $S^k E$ and $\wedge^k E$, respectively. Combining Theorem \ref{thm3} and Proposition \ref{limtcomm}, we conclude
\begin{thm}\label{HNthm}
Let $k\in\mathbb Z_{>0}$. Then
\begin{enumerate}
\item $\vec\mu_{\omega}(E\otimes \tilde E)=\vec T(\vec\mu_\omega(E),\vec\mu_\omega(\tilde E))$;
\item $\vec\mu_\omega(E^{\otimes k})=\vec T_k(\vec\mu_\omega(E))$;
\item $\vec\mu_\omega(S^k E)=\vec S_k(\vec\mu_\omega(E))$;
\item $\vec\mu_\omega(\wedge^kE)=\vec A_k(\vec\mu_\omega(E))$ when $k\leq r$.
\end{enumerate}
\end{thm}

Immediately  Theorem \ref{HNthm} tells us that  for $k\geq 1$ and $l\geq 0$,
\begin{align}
&\mu_L(E^{\otimes k}\otimes \tilde E^{\otimes l},\omega)=k\mu_L(E,\omega)+l\mu_L(\tilde E,\omega),\\
&\mu_U(E^{\otimes k}\otimes \tilde E^{\otimes l},\omega)=k\mu_U(E,\omega)+l\mu_U(\tilde E,\omega),\\
&\mu_{L}(S^kE, \omega )=k\mu_{L}(E, \omega ),\\
&\mu_{U}(S^kE, \omega )=k\mu_{U}(E, \omega ),
\end{align}
and for any $k=1,\cdots,r$,
\begin{align}
&\mu_L(\wedge^k E,\omega)=\sum_{i=1}^k\mu_{r+1-i,\omega}(E)\geq k\mu_L(E,\omega),\\
&\mu_U(\wedge^k E,\omega)=\sum_{i=1}^k\mu_{i,\omega}(E)\leq k\mu_U(E,\omega).
\end{align}

As a simple corollary, we have
\begin{cor}
The following statements are equivalent:
\begin{enumerate}
\item $\mu_L(E,\omega)>0$;
\item $\mu_L(E^{\otimes k},\omega)>0$ for some (resp. every) $k\geq 1$;
\item $\mu_L(S^k E,\omega)>0$ for some (resp. every) $k\geq 1$.
\end{enumerate}
Whenever one of the above holds, $\mu_L(\wedge^k E,\omega)>0$ for every $k=1,\cdots,r$.
\end{cor}

\subsection{A proof of Theorem \ref{ample}}

The ampleness of $(E, \bar{\partial}_{E})$ tells us that $c_{1}(E)>0$ and there exists a Hermitian metric $H_{0}$ such that $\sqrt{-1}\tr F_{H_{0}}>0$. Set the K\"ahler metric $\omega =\sqrt{-1}\tr F_{H_{0}}$. Let $H(t)$ be the solution of the Hermitian-Yang-Mills flow (\ref{Flow}) with initial data $H_{0}$. Noting that $\tr (\sqrt{-1}\Lambda_{\omega } F_{H_{0}}- \textmd{Id}_E)=0$, together with Proposition \ref{Flow002}, one immediately has
\begin{equation}
\sqrt{-1}\tr F_{H(t)} =\sqrt{-1}\tr F_{H_{0}}=\omega
\end{equation}
for any $0\leq t < \infty $. By \cite[Corollary 4.6]{LiZZ}, we have
\begin{equation}
\mu_{L}(E, \omega )> 0.
\end{equation}

To compute the second Chern form, we recall the following well-known formula
\begin{equation}
\begin{split}
& 4\pi^2(2c_2(E, H(t))-\frac{1}{2}c_1^2(E, H(t)))\\
=& (|\sqrt{-1} F_{H(t)}^\perp|^2_{H(t),\omega}-|\sqrt{-1}\Lambda_\omega F_{H(t)}^\perp|_{H(t)}^2)\frac{\omega^{2}}{2},
\end{split}
\end{equation}
where $\sqrt{-1} F_{H(t)}^\perp=\sqrt{-1} F_{H(t)}-\frac{1}{2}\Id_E\otimes\omega$ and $\sqrt{-1}\Lambda_\omega F_{H(t)}^\perp=\sqrt{-1}\Lambda_\omega F_{H(t)}-\Id_E$. Notice that
$4\pi^2c_1^2(E,H(t))=\omega^2$ and
\begin{equation}\begin{aligned}
|\sqrt{-1} F_{H(t)}^\perp|^2_{H(t),\omega}=&\left|\sqrt{-1} F_{H(t)}^\perp-(\sqrt{-1}\Lambda_{\omega}F_{H(t)}^\perp)\otimes \frac{\omega}{2}\right|^2_{H(t)}\\
&+\frac{1}{2}\left|\sqrt{-1}\Lambda_{\omega}F_{H(t)}^\perp \right|^2_{H(t)}\\
\geq & \frac{1}{2}|\sqrt{-1}\Lambda_\omega F_{H(t)}^\perp|_{H(t)}^2.
\end{aligned}\end{equation}
Therefore,
\begin{equation}
8\pi^2 c_2(E, H(t))\geq (2-|\sqrt{-1}\Lambda_\omega F_{H(t)}^\perp|_{H(t)}^2)\frac{\omega^2}{4}.
\end{equation}
As above, we denote the two eigenvalues of the mean curvature $\sqrt{-1}\Lambda_{\omega} F_{H(t)}$  by $\lambda_1(H(t))$, $\lambda_2(H(t))$, sorted in the descending order. It holds that
\begin{equation}
 \lambda_1(H(t))+\lambda_2(H(t))= 2
\end{equation}
and hence
\begin{equation}
\begin{split}
& 2-|\sqrt{-1}\Lambda_\omega F_{H(t)}^\perp|_{H(t)}^2\\
=& 2-\sum_{i=1}^2(\lambda_{i}(H(t))-1)^2\\
=& 4-(\lambda_{1}(H(t)))^2-(\lambda_{2}(H(t)))^2\\
=& (\lambda_{1}(H(t))+\lambda_{2}(H(t)))^2-(\lambda_{1}(H(t)))^2-(\lambda_{2}(H(t)))^2\\
=& 2\lambda_{1}(H(t))\lambda_{2}(H(t)).
\end{split}
\end{equation}
Owing to Theorem \ref{thm3},  we know that when $t$ is large enough, $\lambda_{1}(H(t))\geq\lambda_{2}(H(t))>0$, then
\begin{equation}
c_2(E, H(t))\geq\frac{\lambda_{1}(H(t))\lambda_{2}(H(t))}{16\pi^2}\omega^2>0.
\end{equation}
This completes the proof of Theorem \ref{ample}.

\section{The Atiyah-Bott-Bando-Siu question on some compact non-K\"ahler manifolds}
In this section, we study the Atiyah-Bott-Bando-Siu question on some compact non-K\"ahler manifolds.
\subsection{The  HN-type of the limit}
 We first prove that the HN-type of the limiting sheaf in Theorem \ref{thm1.1} is the same as that of the initial holomorphic vector bundle $(E,\overline{\partial}_{A_0})$.

     In the setting of Theorem \ref{thm1.9}, suppose $A(t)$ is a smooth solution of the heat flow \eqref{modified flow} over $(M,\omega)$ with the initial data $A_0\in \mathcal{A}_{H_0}^{1,1}$, and let $A_{\infty}$ be an Uhlenbeck limit. From Theorem \ref{thm1.1}, we know that $A_{\infty}$ is an admissible connection on the limiting reflexive sheaf $E_{\infty}$, and $\Lambda_{\omega}F_{A_\infty}$ is parallel. Write $r=\mathrm{rank}(E)$ and denote $r$ real eigenvalues of $\sqrt{-1}\Lambda_{\omega}F_{A_\infty}$ by $\lambda_{1, \infty }$, $\lambda_{2, \infty }$, $\cdots$, $\lambda_{r, \infty }$, sorted in the descending order. Then the constant vector $\frac{\mathrm{Vol}(M,\omega)}{2\pi}\vec{\lambda}_{\infty}=(\frac{\mathrm{Vol}(M,\omega)}{2\pi}\lambda_{1, \infty },\cdots,\frac{\mathrm{Vol}(M,\omega)}{2\pi}\lambda_{r, \infty })$ is just the HN-type of $(E_{\infty},\overline{\partial}_{A_\infty})$. Let $\vec{\mu}_{\omega}(E)=(\mu_{1,\omega},\cdots,\mu_{r,\omega})$ be the HN-type of $(E, \bar{\partial}_{A_0})$. Now we are going to prove $\frac{\mathrm{Vol}(M,\omega)}{2\pi}\vec{\lambda}_{\infty}=\vec{\mu}_{\omega}(E)$.

    Let $\mathbf{u}(r)$ denote the Lie algebra of the unitary group $U(r)$. Fix a real number $\rho\geq 1$. For any $\mathbf{a}\in\mathbf{u}(r)$, let $\varphi_{\rho}(\mathbf{a})=\Sigma_{j=1}^r{|\lambda_j|}^{\rho}$, where $\sqrt{-1}\lambda_j$ are the eigenvalues of $\mathbf{a}$. From \cite[Proposition 12.16]{atiyah1983yang} it follows that $\varphi_{\rho}$ is a convex function on $\mathbf{u}(r)$. For a given real number $N$, define the Hermitian-Yang-Mills type functionals as follows:
    \begin{equation}\label{hymfunctional}
        \mathrm{HYM}_{\rho,N}(A)=\int_{M}\varphi_{\rho}(\frac{\mathrm{Vol}(M,\omega)}{2\pi}\Lambda_{\omega}F_{A}-\sqrt{-1}N \Id_E)\frac{\omega^n}{n!},
    \end{equation}
In the following we assume that $\mathrm{Vol}(M,\omega)=2\pi$, and set $\mathrm{HYM}_{\rho,N}(\vec{\mu})=\mathrm{HYM}_{\rho}(\vec{\mu} + \vec{N})=2\pi \varphi_{\rho}(\sqrt{-1}(\vec{\mu} +\vec{N}))$, where $\vec{\mu}+\vec{N}=\diag(\mu_1+N,\cdots,\mu_r+N)$. We need the following two lemmas, whose proofs can be found in \cite[Lemma 2.23 and Proposition 2.24]{DW}.

    \begin{lem}\label{lemma3.1}The functional $\mathbf{a}\mapsto (\int_M \varphi_{\rho}(\mathbf{a})\frac{\omega^n}{n!})^{\frac{1}{\rho}}$ defines a norm on $L^{\rho}(\mathbf{u}(E))$ which is equivalent to the $L^{\rho}$-norm, where $\mathbf{u}(E)$ stands for the subbundle of $\mathrm{End}(E)$ consisting of skew-Hermitian endomorphisms.
    \end{lem}

    \begin{lem}\label{lemma3.2}  Assume $\mu_r\geq 0$ and $\lambda_r \geq 0$. If $\varphi_{\rho}(\sqrt{-1}\vec{\mu})=\varphi_{\rho}(\sqrt{-1}\vec{\lambda})$ for all $\rho$ in some set $I \subset [1,\infty) $ possessing a limit point, then $\vec{\mu}=\vec{\lambda}$.
    \end{lem}

\begin{cor}\label{cor2.6}
Let $A({t}_{i}^{})$ be a sequence of connections along the modified Yang-Mills flow with the Uhlenbeck limit ${A}_{\infty}^{}$. Then\\
$(1)$ $|\Lambda_{\omega}F_{A(t_i)}|_{H_0}\to |\Lambda_{\omega}F_{A_{\infty}}|_{H_0}$ strongly in ${L}_{}^{p}$ for all $1\leq p<\infty$,  and consequently, $\lim_{t_i\to\infty}\int_{M}{\left |\Lambda_{\omega}F_{A(t_i)}\right |}_{H_0}^{2}\frac{\omega^n}{n!}=\int_{M}{\left |\Lambda_{\omega}F_{A_{\infty}}\right |}_{H_0}^{2}\frac{\omega^n}{n!}$;\\
$(2)$ $\left\|\Lambda_{\omega}F_{A_{\infty}}\right\|_{L^{\infty}(H_0)}\leq \left\|\Lambda_{\omega}F_{A(t_j)}\right\|_{L^{\infty}(H_0)}\leq \left\|\Lambda_{\omega}F_{A(t_0)}\right\|_{L^{\infty}(H_0)}$ for $0\leq{t}_{0}\leq{t}_{j}$.
 \end{cor}

    \begin{rem}
     The proof of the above corollary is the same as in K\"ahler case (\cite[Corollary 3.12]{LZ0}).
    \end{rem}

     Note that we can approximate $\varphi_{\rho}$ by smooth convex ad-invariant functions $\varphi_{\rho,\zeta}\to\varphi_{\rho}$. By direct calculation (\cite[Proposition 2.25]{DW}), there holds that
    \begin{equation}\label{102}
            (2\sqrt{-1}\Lambda_{\omega}\partial\overline{\partial}-\frac{\partial}{\partial t})\varphi_{\rho,\zeta}(\Lambda_{\omega}F_{A(t)}-\sqrt{-1}N \mathrm{Id}_E)\geq 0.
    \end{equation}
    So one can see that $t \mapsto \mathrm{HYM}_{\rho,N}(A(t))$ is nonincreasing. Then in conjunction with Lemma \ref{lemma3.1} and Corollary \ref{cor2.6}, this shows that
   \begin{prop}\label{prop3.3}Let $A(t)$ be a solution of the flow \eqref{modified flow} and $A_\infty$ be a subsequential Uhlenbeck limit of $A(t)$. Then for any $\rho \geq 1$ and any $N$, $\mathrm{lim}_{t\to \infty}\mathrm{HYM}_{\rho,N}(A(t))=\mathrm{HYM}_{\rho,N}(A_\infty)$.
    \end{prop}

    Next we are going to show that the HN-type of the limiting sheaf is the same as that of the initial bundle.

    \begin{thm}\label{consistence of HN type}
    Let $(E,\overline{\partial}_{A_0})$ be a holomorphic vector bundle over a compact Gauduchon manifold $(M,\omega)$, and $A(t)$ be the smooth solution of the modified Yang-Mills flow (\ref{modified flow}) on the Hermitian vector bundle $(E,{H}_{0}^{})$ with initial data ${A}_{0}^{}\in \mathcal{A}_{H_0}^{1,1}$. Let ${A}_{\infty}^{}$ be an Uhlenbeck limit of $A(t)$, and $({E}_{\infty}^{},A_{\infty})$ be the corresponding reflexive sheaf as in Theorem \ref{thm1.1}. Then
    \begin{equation}
           {\mathrm{HYM}}_{\rho,N}^{}({A}_{\infty}^{})=\lim_{t \to \infty}{{\mathrm{HYM}}_{\rho,N}^{}(A(t))}={\mathrm{HYM}}_{\rho,N}^{}({\vec{\mu}_{\omega}}(E))
    \end{equation}
    for all $1\leq\rho \leq 2$ and all N $\in\mathbb{R}$; and the HN-type of $({E}_{\infty}^{},{A}_{\infty}^{})$ is the same as that of $(E,{A}_{0}^{})$.
    \end{thm}

    \begin{proof}[Proof of Theorem \ref{consistence of HN type}]
    Suppose $H(t)$ is the solution of the Hermitian-Yang-Mills flow (\ref{Flow}) with the initial condition $H_0$ and $D_{\overline{\partial}_{A_0},H(t)}$ is the Chern connection on $(E,\overline{\partial}_{A_0})$ with respect to $H(t)$. Fix $1\leq\rho\leq 2$, obviously there holds
    \begin{equation}\label{ineqHYM}
    \begin{split}
          &{\left| \mathrm{HYM}_{\rho,N}(\vec{\lambda}_{\infty})^{\frac{1}{\rho}}-\mathrm{HYM}_{\rho,N}(\vec{\mu}_{\omega}(E))^{\frac{1}{\rho}}\right|}\\
          \leq & {\left| \mathrm{HYM}_{\rho,N}(A_{\infty})^{\frac{1}{\rho}}-\mathrm{HYM}_{\rho,N}(A(t))^{\frac{1}{\rho}}\right|}\\
               &+{\left|\mathrm{HYM}_{\rho,N}(D_{\overline{\partial}_{A_0},H(t)})^{\frac{1}{\rho}}-\mathrm{HYM}_{\rho,N}(\vec{\mu}_{\omega}(E))^{\frac{1}{\rho}}\right|},\\
    \end{split}
    \end{equation}where we have used the fact
    \begin{equation}
           \mathrm{HYM}_{\rho,N}(A_{\infty})=\mathrm{HYM}_{\rho,N}(\vec{\lambda}_{\infty})\quad \text{and} \quad \mathrm{HYM}_{\rho,N}(A(t))=\mathrm{HYM}_{\rho,N}(D_{\overline{\partial}_{A_0},H(t)}).
    \end{equation} Note that
    \begin{equation}
    \begin{split}
           &{\left|\mathrm{HYM}_{\rho,N}(D_{\overline{\partial}_{A_0},H(t)})^{\frac{1}{\rho}}-\mathrm{HYM}_{\rho,N}(\vec{\mu}_{\omega}(E))^{\frac{1}{\rho}}\right|}\\
           = & \left|\left(\int_M \varphi_\rho(\sqrt{-1}(\vec{\lambda}(H(t),\omega)+\vec{N}))\frac{\omega^n}{n!}\right)^{\frac{1}{\rho}}-\left(\int_M \varphi_\rho(\sqrt{-1}(\vec{\mu}_\omega(E)+\vec{N}))\frac{\omega^n}{n!}\right)^{\frac{1}{\rho}}\right|\\
           \leq & \left(\int_M\varphi_\rho(\sqrt{-1}(\vec{\lambda}(H(t),\omega)-\vec{\mu}_\omega(E)))\frac{\omega^n}{n!}\right)^{\frac{1}{\rho}}\\
           \leq & C(\rho,r)\|\vec{\lambda}(H(t),\omega)-\vec{\mu}_\omega(E)\|_{L^\rho(M,\omega)},\\
    \end{split}
    \end{equation}
 where $C(\rho,r)$ is a constant depending only on $r$ and $\rho$. For any $\delta>0$, by Theorem \ref{thm3}, we get
    \begin{equation}
    \begin{split}
           &{\left|\mathrm{HYM}_{\rho,N}(D_{\overline{\partial}_{A_0},H(t)})^{\frac{1}{\rho}}-\mathrm{HYM}_{\rho,N}(\vec{\mu}_{\omega}(E))^{\frac{1}{\rho}}\right|}\leq \frac{\delta}{2}\\
    \end{split}
    \end{equation}when $t$ is sufficiently large. Proposition \ref{prop3.3} gives us that
    \begin{equation}
           {\left| \mathrm{HYM}_{\rho,N}(A_{\infty})^{\frac{1}{\rho}}-\mathrm{HYM}_{\rho,N}(A(t))^{\frac{1}{\rho}}\right|}\leq \frac{\delta}{2}
    \end{equation}when $t$ is large enough and therefore
    \begin{equation}
           {\left| \mathrm{HYM}_{\rho,N}(\vec{\lambda}_{\infty})^{\frac{1}{\rho}}-\mathrm{HYM}_{\rho,N}(\vec{\mu}_{\omega}(E))^{\frac{1}{\rho}}\right|}\leq \delta.
    \end{equation}Taking $\delta\to 0$, we deduce
    \begin{equation}
           \mathrm{HYM}_{\rho,N}(\vec{\lambda}_{\infty})=\mathrm{HYM}_{\rho,N}(\vec{\mu}_{\omega}(E))
    \end{equation}for any $1\leq\rho \leq 2$ and all $N\in\mathbb{R}$. Of course, this says that $\varphi_\rho(\sqrt{-1}(\vec{\lambda}_\infty+\vec{N}))=\varphi_\rho(\sqrt{-1}(\vec{\mu}_\omega(E)+\vec{N}))$ for all $1\leq \rho\leq 2$ and all $N\in\mathbb{R}$. We can choose $N$ large enough such that $\mu_{r,\omega}+N\geq 0$ and $\lambda_{r,\infty}+N\geq 0$. Then due to Lemma \ref{lemma3.2}, one can conclude $\vec{\lambda}_\infty=\vec{\mu}_\omega(E)$.
    \end{proof}

Let $H(t)$ be the long time solution of the Hermitian-Yang-Mills flow (\ref{Flow}) with the initial metric $H_0$ and $A(t)$ the long time solution of the modified Yang-Mills flow (\ref{modified flow}) with $A_0\in \mathcal{A}_{H_0}^{1,1}$. Then we know $A(t)= \sigma(t)(A_0)$, where $\sigma(t)$ is a family of complex gauge transformations satisfying $\sigma^{\ast_{H_0}}(t)\sigma(t)= H_0^{-1}H(t)$. Consider the following HN-filtration of $(E,\overline{\partial}_{A_0})$ by saturated subsheaves
\begin{equation}
0=\mathcal{E}_0\subset \mathcal{E}_1\subset \cdots \subset \mathcal{E}_l=(E,\overline{\partial}_{A_0}).
\end{equation}
Suppose $\kappa_{\alpha}^{H(t)}$ is the orthogonal projection onto $\mathcal{E}_\alpha$ with respect to $H(t)$ and set $\kappa_{\alpha}^{(t)}=\sigma(t)\circ \kappa_{\alpha}^{H(t)} \circ \sigma^{-1}(t)$. Then $\kappa_{\alpha}^{(t)}$ is the orthogonal projection onto the subsheaf $\sigma(t)(\mathcal{E}_\alpha)$ and $\{\kappa_{\alpha}^{(t)}\}$ is the HN-filtration of $(E,\overline{\partial}_{A(t)})$. It is easy to check that: $(\kappa_{\alpha}^{(t)})^2=\kappa_{\alpha}^{(t)}=(\kappa_{\alpha}^{(t)})^{\ast_{H_0}}$; $(\Id - \kappa_{\alpha}^{(t)})\overline{\partial}_{A(t)}\kappa_{\alpha}^{(t)}=0$; $|\overline{\partial}_{A(t)}\kappa_{\alpha}^{(t)}|_{H_0}=|\overline{\partial}_{A_0}\kappa_{\alpha}^{H(t)}|_{H(t)}$. Moreover, one can see that $\kappa_{\alpha}^{(t)}\in L_1^2(\End(E))$.

Under the same condition as in Theorem \ref{thm1.1}, assuming $A_{\infty}$ is the Uhlenbeck limit of $A(t_j)$, we know there is a sequence of unitary gauge transformations $\{a_j\}$ such that the connections $A_j:=a_j(A(t_j))$ converge to $A_{\infty}$ in $C^{\infty}_{loc}$-topology off $\Sigma_{an}$ as $j\to \infty$. Then $\kappa^{j}_{\alpha}:=a_j\circ \kappa^{(t_j)}_{\alpha}\circ a_j^{-1}$ is the orthogonal projection onto $a_j\circ\sigma(t_j)(\mathcal{E}_{\alpha})$ with respect to $H_0$. By \cite[Lemma 3.13]{LZ1} (also see \cite[Lemma 4.5]{DW}), we have the following lemma.
    \begin{lem}\label{lemma3.4}  Suppose $({E}_{\infty}, A_{\infty})$ is the corresponding reflexive sheaf as in Theorem \ref{thm1.1}.
                          \\ $(1)$ Let $\{ {\kappa}_{\alpha}^{\infty} \}$ be the HN-filtration of the reflexive sheaf $({E}_{\infty}^{},\overline{\partial}_{A_\infty})$. Then there is a subsequence of HN-filtration $\{ {\kappa}_{\alpha}^{j} \}$ converging to $\{ {\kappa}_{\alpha}^{\infty} \}$ strongly in ${L}^{p}\cap{L}_{1,loc}^{2}$ off ${\Sigma}_{an}$ as $j\rightarrow +\infty$ for all $1 \leq p < \infty$.
                          \\ $(2)$ Assume $(E,\overline{\partial}_{A_0})$ is semi-stable and $\{\kappa_{ss,\alpha}^{j}\}$ is a Seshadri filtration of $(E,\overline{\partial}_{A_j})$.Without loss of generality, suppose the ranks of $\kappa_{ss,\alpha}^{j}$ are constant for all $j$. Then there is a filtration $\{ {\kappa}_{ss,\alpha}^{\infty} \}$ of $({E}_{\infty},\overline{\partial}_{A_\infty})$ such that  after passing to a subsequence, $\{\kappa_{ss,\alpha}^{j}\}$ converges to $\{ {\kappa}_{ss,\alpha}^{\infty} \}$ strongly in ${L}^{p}\cap{L}_{1,loc}^{2}$ off ${\Sigma}_{an}$ as $j\rightarrow +\infty$ for all $1 \leq p < \infty$, the rank and degree of ${\kappa}_{ss,\alpha}^{\infty}$ are equal to the rank and degree of ${\kappa}_{ss,\alpha}^{j}$ for all $\alpha$ and j.
    \end{lem}
\subsection{The non-zero holomorphic map}

   In this subsection, we show the existence of a non-zero holomorphic map under some assumptions, which is crucial to construct the nonzero holomorphic map from the graded sheaf to the limiting sheaf.

Assume that $\pi:\tilde{M}\to M$ is a blow-up with smooth center $\Sigma$. From \cite{Vhodge}, one can see there exists a holomorphic line bundle $L$ over $\tilde{M}$ with respect to the divisor $-\pi^{-1}(\Sigma)$ such that the $(1,1)$-form $\pi^{\ast}\omega+\tilde{\delta}\sqrt{-1}F_{H_L}$ is positive for some small $\tilde{\delta}$, where $\sqrt{-1}F_{H_L}$ is the Chern form with respect to some Hermitian metric $H_L$ on $L$. Set $\eta=\pi^{\ast}\omega+\tilde{\delta}\sqrt{-1}F_{H_L}$ and $\omega_{\epsilon}=\pi^{\ast}\omega+\epsilon\eta$ for $0<\epsilon\leq 1$. Then we get a family of Hermitian metrics $\{\omega_\epsilon\}$ on $\tilde{M}$. Moreover, if $\omega$ is K\"ahler, then $\eta$ is K\"ahler and so is $\omega_\epsilon$. For a finite sequence of blow-ups $\pi_i:M_i\to M_{i-1}$ $(i=1,2,\cdots,q)$ with $M_0=M$, set
\begin{equation}\label{construction 2.2}
   \omega_{i,\epsilon}=\pi_i^{\ast}\omega_{i-1,\epsilon}+\epsilon_i\eta_i,
\end{equation}
where each $\eta_i$ is a Hermitian metric on $M_i$, $0<\epsilon_i\leq 1$ and $\omega_{0,\epsilon}=\omega$.

Bando and Siu (\cite{BS}) derived a uniform Sobolev inequality for a blow-up with the smooth center. Their argument does not rely on the K\"ahler condition. We include the result here for completeness.

\begin{lem}[\cite{BS}]\label{lemma2.1}
Let $(M,\omega)$ be an $n$-dimensional compact Hermitian manifold and $\pi:\tilde{M}\to M$ a blow-up with non-singular center. Fix an arbitrary Hermitian metric $\theta$ on $\tilde{M}$ and set $\omega_{\epsilon}=\pi^{\ast}\omega+\epsilon\theta$ for $0<\epsilon\leq 1$. Then there exists a uniform constant $C_S(\omega)$ such that
\begin{equation}
(\int_{\tilde{M}}|f|^{\frac{2n}{2n-1}}\frac{\omega_{\epsilon}^n}{n!})^{\frac{2n-1}{n}}\leq C_S(\omega)(\int_{\tilde{M}}(|\mathrm{d}f|_{\omega_{\epsilon}}^2+|f|^2)\frac{\omega_{\epsilon}^n}{n!}),
\end{equation}for all $f\in C^1(\tilde{M})$ and all $0<\epsilon\leq 1$.
\end{lem}
Now we present the following two lemmas without proof, as these results can be derived directly via the Moser iteration.
\begin{lem}\label{lemma4.2}
Let $(M,\omega)$ be an $n$-dimensional compact Hermitian manifold. Suppose the function $\phi\geq 0$ satisfies $\phi(x,t)\in C^{\infty}(M\setminus\Sigma\times[0,\infty))\cap L^{\infty}(M\setminus\Sigma\times[t_1,t_2])$ for arbitrary $t_2>t_1>0$, where $\Sigma$ is an analytic subset of complex codimension at least $2$. Assume $(2\sqrt{-1}\Lambda_\omega\partial\overline{\partial}-\frac{\partial}{\partial t})\phi\geq 0$, $\phi(0)\in L^1(M\setminus\Sigma,\omega)$ and $\|\phi(t)\|_{L^1(M\setminus\Sigma,\omega)}\leq e^{C_1(t-t_0)}\|\phi(t_0)\|_{L^1(M\setminus\Sigma,\omega)}$ for any $t\geq t_0\geq 0$, where $C_1$ is a positive constant. Then for any $T>s>0$,
\begin{equation}
\|\phi\|_{L^{\infty}(M\setminus\Sigma\times[s,T])}\leq\frac{C(C_S(\omega),n,T,C_1)}{s^{2n+1}}\|\phi(0)\|_{L^1(M\setminus\Sigma)},
\end{equation}
where $C(C_S(\omega),n,T,C_1)$ is a constant depending only on $n$, $T$, $C_1$ and the Sobolev constant $C_S(\omega)$.
\end{lem}
\begin{lem}\label{lemma4.3}
Let $(M,\omega)$ be an $n$-dimensional compact Hermitian manifold, $\vartheta$ a Hermitian metric such that $\vartheta\geq \omega$. Suppose the function $\phi(x,t)\in C^{\infty}(M\setminus\Sigma\times[0,\infty))$, where $\Sigma$ is a closed subset in $M$. Assume $\vartheta$ and $\omega$ are quasi-isometric for arbitrary $\Omega\subset\subset M\setminus\Sigma$ (i.e. there exists a constant $C(\Omega)$ with $0< C(\Omega)< 1$ such that $C(\Omega)\vartheta\leq\omega\leq\vartheta$ on $\Omega$). Assume $(2\sqrt{-1}\Lambda_\omega\partial\overline{\partial}-\frac{\partial}{\partial t})\phi\geq 0$, $\phi(t)\in L^1(M\setminus\Sigma,\omega)$ and $\|\phi(t)\|_{L^1(M\setminus\Sigma,\omega)}\leq e^{C_1(t-t_0)}\|\phi(t_0)\|_{L^1(M\setminus\Sigma,\omega)}$ for any $t\geq t_0\geq 0$. Then for any $\hat{\delta}>0$,
\begin{equation}
\begin{split}
&\|\phi\|_{L^{\infty}(M\setminus B_{2\hat{\delta}}^{\vartheta}(\Sigma)\times[0,T])}\\
\leq &\frac{C(C_S(\omega),n,T,C_1,C_2,\hat{\delta},\mathrm{Vol}(M,\omega))}{\hat{\delta}^{4n+2}}(\|\phi(0)\|_{L^1(M\setminus\Sigma,\omega)}+\|\phi(0)\|_{L^{\infty}(M\setminus B^{\vartheta}_{\hat{\delta}}(\Sigma))}),\\
\end{split}
\end{equation}where $B_{\hat{\delta}}^{\vartheta}(\Sigma)$ denotes the $\hat{\delta}$-neighborhood of $\Sigma$ with respect to $\vartheta$, $C_2$ is the constant such that $-C_2\frac{\omega^n}{n!}\leq\frac{\sqrt{-1}\partial\overline{\partial}\omega^{n-1}}{(n-1)!}\leq C_2\frac{\omega^n}{n!}$ and $|\mathrm{d}\omega|_\omega\leq C_2$.
\end{lem}

Let $(M,\omega)$ be a Gauduchon manifold, $\mathcal{E}$ be a torsion-free sheaf with the Hermitian metric $H_0$ outside the singular set $\Sigma_{\mathcal{E}}$. Thanks to Hironaka's flattening theorem (\cite{Hi2}), we know there exists a finite sequence of blow-ups $\pi_i:M_i\to M_{i-1}(i=1,\cdots,q$ and $M_0=M)$ along the smooth center $\Sigma_{i-1}$ of complex codimension at least $2$ such that\\
(1)$E_q=\pi^{\ast}\mathcal{E}/tor(\pi^{\ast}\mathcal{E})$ is locally free over $M_q=\tilde{M}$, where $\pi=\pi_1\circ\pi_2\circ\cdots\circ\pi_q$ and $tor(\pi^{\ast}\mathcal{E})$ is the torsion part of $\pi^{\ast}\mathcal{E}$;\\
(2)$\pi:\tilde{M}\to M$ is biholomorphic restricted to $\tilde{M}\setminus\pi^{-1}(\Sigma_{\mathcal{E}})$.

\begin{prop}\label{lemma2.2}Suppose $\pi^{\ast}H_0$ can be extended smoothly to the whole $M_q$ and still denote it by $H_0$ for simplicity. Then there exists a long-time solution $H(t)$ to Donaldson's heat flow
\begin{equation}
    \left\{\begin{aligned}\label{inequalty2.5}
           H^{-1}(t)\frac{\partial H(t)}{\partial t}&=-2\sqrt{-1}\Lambda_{\omega}F_{H(t)},\\
           H(0)&=H_0,\\
                                  \end{aligned}
                                  \right.
                           \end{equation}on $M\setminus\Sigma_{\mathcal{E}}$ satisfying
\begin{equation}\label{relationofL1}
\int_{M\setminus\Sigma_{\mathcal{E}}}|\sqrt{-1}\Lambda_\omega F_{H(t)}|_{H(t)}\frac{\omega^n}{n!}\leq e^{2\tilde{C}_1(t-t_0)}\int_{M\setminus\Sigma_{\mathcal{E}}}|\sqrt{-1}\Lambda_\omega F_{H(t_0)}|_{H(t_0)}\frac{\omega^n}{n!},
\end{equation}
\begin{equation}\label{inequalty2.7}
\begin{split}
&\int_{M\setminus\Sigma_{\mathcal{E}}}\ln(\mathrm{tr} h(t)+\mathrm{tr} h^{-1}(t))\frac{\omega^n}{n!}\\
\leq &e^{2\tilde{C}_1 t}\ln(2\mathrm{rank}(\mathcal{E}))\mathrm{Vol}(M,\omega)+\frac{e^{2\tilde{C}_1 t}-1}{2\tilde{C}_1}\int_{M\setminus\Sigma_{\mathcal{E}}}|\Lambda_{\omega}F_{H_0}|_{H_0}\frac{\omega^n}{n!},
\end{split}
\end{equation}and
\begin{equation}\label{inequalty2.8}
|\sqrt{-1}\Lambda_{\omega}F_{H}|_{L^{\infty}(M\setminus\Sigma_{\mathcal{E}}\times[s,T],H)}\leq\frac{C(C_S(\omega),n,\tilde{C}_1,T)}{s^{2n+1}}\int_{M\setminus\Sigma_{\mathcal{E}}}|\sqrt{-1}\Lambda_\omega F_{H_0}|_{H_0}\frac{\omega^n}{n!}
\end{equation}for any $t\geq t_0\geq 0$ and any $T>s>0$, where $h(t)=H_0^{-1}H(t)$ and $\tilde{C}_1$ is a constant depending only on $\omega$.
\end{prop}
\begin{proof}
Define the $i$th blow-up $\pi_i:M_i\to M_{i-1}$ along $\Sigma_{i-1}$ for $1\leq i\leq q$. For the $q$th blow-up, $\pi_q$ is biholomorphic between $M_q\setminus\pi_q^{-1}(\Sigma_{q-1})$ and $M_{q-1}\setminus\Sigma_{q-1}$. Then $E_{q-1}:={(\pi_q)}_{\ast}E$ is locally free outside $\Sigma_{q-1}$. Set $E_{i-1}=(\pi_i)_{\ast}E_i$, $\hat{\Sigma}_q=\emptyset$ and $\hat{\Sigma}_{i-1}=\pi_i(\hat{\Sigma}_i)\cup\Sigma_{i-1}$ for $1\leq i\leq q$. That is, there are $q$ times blowups such that $E_i$ over $M_i\setminus \pi_i^{-1}(\Sigma_{i-1})$ is isomorphic to $E_{i-1}$ over $M_{i-1}\setminus\Sigma_{i-1}$ and $E_{i-1}$ is locally free over $M_{i-1}\setminus\hat{\Sigma}_{i-1}$.
$$
\def\arl{\ar@{->}}
\xymatrix{
  E_q\arl[d] & E_{q-1}\arl[d] & E_{q-2}\arl[d] &              & E_1\arl[d] & E_0\arl[d]\\
  M_q\arl[r]^{\pi_q} & M_{q-1}\arl[r]^{\pi_{q-1}}& M_{q-2}\arl[r] & \cdots\arl[r] & M_1\arl[r] & M\\
}
$$
Because $\pi_i$ is holomorphic and proper, by Remmert's proper mapping theorem (\cite{RM56},\cite{RM57}), $\hat{\Sigma}_{i-1}\subset M_{i-1}$ is an analytic subset of complex codimension at least $2$. On the other hand, by the assumption (2), $E_0$ is isomorphic to $\mathcal{E}$ outside $\Sigma_{\mathcal{E}}$.

\begin{claim}\label{claim2.1} There exists a constant $\tilde{C}_1$ independent of $\epsilon_i$ for every $0<\epsilon_i\leq 1$ and $1\leq i\leq q$ such that
\begin{equation}
-\tilde{C}_1\frac{\omega_{q,\epsilon}^n}{n!}\leq\frac{\sqrt{-1}\partial\overline{\partial}\omega_{q,\epsilon}^{n-1}}{(n-1)!}\leq \tilde{C}_1\frac{\omega_{q,\epsilon}^n}{n!}\quad and\quad |\mathrm{d}\omega_{q,\epsilon}|_{\omega_{q,\epsilon}}\leq \tilde{C}_1,
\end{equation}
where $\omega_{q, \epsilon}$ is defined by (\ref{construction 2.2}).
\end{claim}
\begin{proof}[Proof of Claim \ref{claim2.1}] Recalling the definition and computing directly, we have
\begin{equation}\label{partialopomega}
\begin{split}
&\sqrt{-1}\partial\overline{\partial}\omega_{q,\epsilon}^{n-1}\\
=& (n-2)(n-1)\sqrt{-1}\partial\omega_{q,\epsilon}\wedge\overline{\partial}\omega_{q,\epsilon}\wedge\omega_{q,\epsilon}^{n-3}+(n-1)\sqrt{-1}\partial\overline{\partial}\omega_{q,\epsilon}\wedge\omega_{q,\epsilon}^{n-2},\\
\end{split}
\end{equation}and
\begin{equation}
\begin{split}
&\sqrt{-1}\partial\omega_{q,\epsilon}\wedge\overline{\partial}\omega_{q,\epsilon}\\
=&(1+\epsilon_q)^2\pi_q^{\ast}(\sqrt{-1}\partial\omega_{q-1,\epsilon}\wedge\overline{\partial}\omega_{q-1,\epsilon})\\
=&(1+\epsilon_q)^2(1+\epsilon_{q-1})^2\pi_q^{\ast}\pi_{q-1}^{\ast}(\sqrt{-1}\partial\omega_{q-2,\epsilon}\wedge\overline{\partial}\omega_{q-2,\epsilon})\\
=&\cdots=(1+\epsilon_1)^2(1+\epsilon_2)^2\cdots(1+\epsilon_q)^2\pi^{\ast}(\sqrt{-1}\partial\omega\wedge\overline{\partial}\omega).\\
\end{split}
\end{equation}Since $M$ is compact, there exists a positive constant $\hat{C}_1$ such that
\begin{equation}\label{11111}
-\hat{C}_1\omega^3\leq\sqrt{-1}\partial\omega\wedge\overline{\partial}\omega\leq \hat{C}_1\omega^3
\end{equation}on $M$. Here, for any $\beta_1,\beta_2\in\Omega_{x}^{p,p}M$, $\beta_1\leq\beta_2$ means that $\beta_2-\beta_1$ is a nonnegative $(p,p)$-form. Because $\pi:M_q\setminus\pi^{-1}(\Sigma_{\mathcal{E}})\to M\setminus\Sigma_{\mathcal{E}}$ is biholomorphic, (\ref{11111}) implies that
\begin{equation}
-\hat{C}_1(\pi^{\ast}\omega)^3\wedge\omega_{q,\epsilon}^{n-3}\leq\sqrt{-1}\partial(\pi^{\ast}\omega)\wedge\overline{\partial}(\pi^{\ast}\omega)\wedge\omega_{q,\epsilon}^{n-3}\leq \hat{C}_1(\pi^{\ast}\omega)^3\wedge\omega_{q,\epsilon}^{n-3}
\end{equation}on $M_q\setminus\pi^{-1}(\Sigma_\mathcal{E})$. From the definition of $\omega_{q,\epsilon}$, we know that $\pi^{\ast}\omega\leq\omega_{q,\epsilon}$, and then
\begin{equation}
-\hat{C}_1\omega_{q,\epsilon}^n\leq\sqrt{-1}\partial(\pi^{\ast}\omega)\wedge\overline{\partial}(\pi^{\ast}\omega)\wedge\omega_{q,\epsilon}^{n-3}\leq \hat{C}_1\omega_{q,\epsilon}^n
\end{equation}on $M_q\setminus\pi^{-1}(\Sigma_\mathcal{E})$. Notice that the codimension of $\pi^{-1}(\Sigma_\mathcal{E})$ is one. By continuity, one can obtain
\begin{equation}\label{2.14ineq}
-\hat{C}_1\omega_{q,\epsilon}^n\leq\sqrt{-1}\partial(\pi^{\ast}\omega)\wedge\overline{\partial}(\pi^{\ast}\omega)\wedge\omega_{q,\epsilon}^{n-3}\leq \hat{C}_1\omega_{q,\epsilon}^n
\end{equation}on $M_q$. Similarly, there exists a constant $\breve{C}_1$ independent of $\epsilon_1,\cdots,\epsilon_q$ such that
\begin{equation}\label{2.15ineq}
-\breve{C}_1\omega_{q,\epsilon}^n\leq \sqrt{-1}\partial\overline{\partial}(\pi^{\ast}\omega)\wedge\omega_{q,\epsilon}^{n-2}\leq \breve{C}_1\omega_{q,\epsilon}^n.
\end{equation}On the other hand, direct calculation shows that
\begin{equation}
\mathrm{d}\omega_{q,\epsilon}=(1+\epsilon_1)(1+\epsilon_2)\cdots(1+\epsilon_q)\pi^{\ast}(\mathrm{d}\omega).
\end{equation}The fact that $\omega_{q,\epsilon}>\pi^{\ast}\omega$ yields that
\begin{equation}\label{2.17ineq}
|\mathrm{d}\omega_{q,\epsilon}|_{\omega_{q,\epsilon}}\leq(1+\epsilon_1)(1+\epsilon_2)\cdots(1+\epsilon_q)|\mathrm{d}\omega|_{\omega}\leq \check{C}_1
\end{equation}for a positive constant $\check{C}_1$ independent of $\epsilon_1,\cdots,\epsilon_q$. Combining (\ref{2.14ineq}), (\ref{2.15ineq}) and (\ref{2.17ineq}), we confirm the claim.
\end{proof}

Now we prove (\ref{inequalty2.5}), (\ref{relationofL1}), (\ref{inequalty2.7}) and (\ref{inequalty2.8}) by induction. Since $E$ is a holomorphic vector bundle over the compact Hermitian manifold $(M_q,\omega_{q,\epsilon})$, by the result in \cite{Z05}, there exists a long time solution $H_{q,\epsilon}(t)$ of Donaldson's heat flow
\begin{equation}
    \left\{\begin{aligned}\label{inequalty2.19}
           H_{q,\epsilon}^{-1}(t)\frac{\partial H_{q,\epsilon}(t)}{\partial t}&=-2\sqrt{-1}\Lambda_{\omega_{q,\epsilon}}F_{H_{q,\epsilon}(t)},\\
           H_{q,\epsilon}(0)&=H_0\\
                                  \end{aligned}
                                  \right.
                           \end{equation}on $M_q$. Direct computation shows that
\begin{equation}\label{eq2.20}
(2\sqrt{-1}\Lambda_{\omega_{q,\epsilon}}\partial\overline{\partial}-\frac{\partial}{\partial t})(|\sqrt{-1}\Lambda_{\omega_{q,\epsilon}}F_{H_{q,\epsilon}(t)}|_{H_{q,\epsilon}(t)}^2+\zeta)^{\frac{1}{2}}\geq 0,
\end{equation}
for any $\zeta >0$.
Integrating both sides of (\ref{eq2.20}) over $(M_q,\omega_{q,\epsilon})$, we obtain
\begin{equation}
\begin{split}
&\frac{\partial}{\partial t}\int_{M_q}(|\sqrt{-1}\Lambda_{\omega_{q,\epsilon}}F_{H_{q,\epsilon}(t)}|^2_{H_{q,\epsilon}(t)}+\zeta)^{\frac{1}{2}}\frac{\omega_{q,\epsilon}^n}{n!}\\
\leq & 2\int_{M_q}(|\sqrt{-1}\Lambda_{\omega_{q,\epsilon}}F_{H_{q,\epsilon}(t)}|^2_{H_{q,\epsilon}(t)}+\zeta)^{\frac{1}{2}}\frac{\sqrt{-1}\partial\overline{\partial}\omega_{q,\epsilon}^{n-1}}{(n-1)!}\\
\leq & 2\tilde{C}_1\int_{M_q}(|\sqrt{-1}\Lambda_{\omega_{q,\epsilon}}F_{H_{q,\epsilon}(t)}|^2_{H_{q,\epsilon}(t)}+\zeta)^{\frac{1}{2}}\frac{\omega_{q,\epsilon}^n}{n!}.\\
\end{split}
\end{equation}Integrating over $[t_0,t]$ for any $t\geq t_0\geq 0$, and letting $\zeta\to 0$, one can obtain
\begin{equation}\label{L1meancurv}
\int_{M_q}|\sqrt{-1}\Lambda_{\omega_{q,\epsilon}}F_{H_{q,\epsilon}(t)}|_{H_{q,\epsilon}(t)}\frac{\omega_{q,\epsilon}^n}{n!}\leq e^{2\tilde{C}_1(t-t_0)}\int_{M_q}|\sqrt{-1}\Lambda_{\omega_{q,\epsilon}}F_{H_{q,\epsilon}(t_0)}|_{H_{q,\epsilon}(t_0)}\frac{\omega_{q,\epsilon}^n}{n!}
\end{equation}Together with (\ref{L1meancurv}), Lemma \ref{lemma4.2} and Lemma \ref{lemma4.3} tell us that
\begin{equation}\label{Linftybound}
|\sqrt{-1}\Lambda_{\omega_{q,\epsilon}}F_{H_{q,\epsilon}(t)}|_{L^{\infty}(M_q\times [s,T],H_{q,\epsilon}(t))}\leq\frac{C(C_S(\omega_{q,\epsilon}),n,\tilde{C}_1,T)}{s^{2n+1}}\int_{M_q}|\sqrt{-1}\Lambda_{\omega_{q,\epsilon}}F_{H_0}|_{H_0}\frac{\omega_{q,\epsilon}^n}{n!}
\end{equation}for any $T>s>0$ and
\begin{equation}
|\sqrt{-1}\Lambda_{\omega_{q,\epsilon}}F_{H_{q,\epsilon}(t)}|_{L^{\infty}(M_q\setminus B_{\check{\delta}}^{\omega_{q,\epsilon}}(\pi^{-1}_q(\Sigma_{q-1}))\times [0,T],H_{q,\epsilon}(t))}\leq C(\check{\delta}^{-1},T)
\end{equation}
for any $\check{\delta}>0$, where $C(\check{\delta}^{-1},T)$ is a uniform constant in $0<\epsilon_q\leq 1$. Set $h_{q,\epsilon}(t)=H_0^{-1}H_{q,\epsilon}(t)$. Straightforward computation shows that
\begin{equation}
(2\sqrt{-1}\Lambda_{\omega_{q,\epsilon}}\partial\overline{\partial}-\frac{\partial}{\partial t})\ln(\mathrm{tr} h_{q,\epsilon}(t)+\mathrm{tr} h_{q,\epsilon}^{-1}(t))\geq-2|\Lambda_{\omega_{q,\epsilon}}F_{H_0}|_{H_0}.
\end{equation}Then there holds that
\begin{equation}\label{tr+tr-1}
\begin{split}
&\int_{M_q}\ln(\mathrm{tr} h_{q,\epsilon}(t)+\mathrm{tr} h_{q,\epsilon}^{-1}(t))\frac{\omega_{q,\epsilon}^n}{n!}\\
\leq & e^{2\tilde{C}_1 t}\ln(2\mathrm{rank}(\mathcal{E}))\mathrm{Vol}(M_q,\omega_{q,\epsilon})+\frac{e^{2\tilde{C}_1 t}-1}{2\tilde{C}_1}\int_{M_q}|\Lambda_{\omega_{q,\epsilon}}F_{H_0}|_{H_0}\frac{\omega_{q,\epsilon}^n}{n!}.\\
\end{split}
\end{equation}

 Note that the constant $C(C_S(\omega_{q,\epsilon}),n,\tilde{C}_1,T)$ in (\ref{Linftybound}) does not depend on $\epsilon_q$. In fact, fixing $(\epsilon_1,\epsilon_2,\cdots\epsilon_{q-1})$, due to Lemma \ref{lemma2.1}, we know the Sobolev constant $C_S(\omega_{q,\epsilon})$ is bounded uniformly in $0<\epsilon_q\leq 1$. Applying the maximum principle and standard elliptic theory, by the same argument as in \cite{LZZ}, we get the uniform $C^\infty_{loc}$-estimate of $h_{q,\epsilon}(t)$ on $M_q\setminus\pi^{-1}_q(\Sigma_{q-1})$. Then by choosing a subsequence, $H_{q,\epsilon}(t)$ converges to $H_{q-1,\epsilon}(t)$ in $C_{loc}^\infty$-topology as $\epsilon_q\to 0$ and $H_{q-1,\epsilon}(t)$ is the solution of Donaldson's heat flow on $M_{q-1}\setminus\Sigma_{q-1}\times[0,\infty]$. Then taking $\epsilon_q\to 0$ on both sides of (\ref{L1meancurv}), (\ref{tr+tr-1}) and (\ref{Linftybound}), one can deduce
\begin{equation}\label{L1meancurva}
\begin{split}
&\int_{M_{q-1}\setminus\Sigma_{q-1}}|\sqrt{-1}\Lambda_{\omega_{q-1,\epsilon}}F_{H_{q-1,\epsilon}(t)}|_{H_{q-1,\epsilon}(t)}\frac{\omega_{q-1,\epsilon}^n}{n!}\\
\leq &e^{2\tilde{C}_1(t-t_0)}\int_{M_{q-1}\setminus\Sigma_{q-1}}|\sqrt{-1}\Lambda_{\omega_{q-1,\epsilon}}F_{H_{q-1,\epsilon}(t_0)}|_{H_{q-1,\epsilon}(t_0)}\frac{\omega_{q-1,\epsilon}^n}{n!}
\end{split}
\end{equation}for any $t\geq t_0$,
\begin{equation}\label{0.61}
\begin{split}
&\int_{M_{q-1}\setminus\Sigma_{q-1}}\ln(\mathrm{tr} h_{q-1,\epsilon}(t)+\mathrm{tr} h_{q-1,\epsilon}^{-1}(t))\frac{\omega_{q-1,\epsilon}^n}{n!}\\
\leq& e^{2\tilde{C}_1 t}\ln(2\mathrm{rank}(\mathcal{E}))\mathrm{Vol}(M_{q-1},\omega_{q-1,\epsilon})+\frac{e^{2\tilde{C}_1 t}-1}{2\tilde{C}_1}\int_{M_{q-1}\setminus\Sigma_{q-1}}|\Lambda_{\omega_{q-1,\epsilon}}F_{H_0}|_{H_0}\frac{\omega_{q-1,\epsilon}^n}{n!}\\
\end{split}
\end{equation}and
\begin{equation}\label{Linftymeancurv1}
|\sqrt{-1}\Lambda_{\omega_{q-1,\epsilon}}F_{H_{q-1,\epsilon}}|_{H_{q-1,\epsilon}}\in L^{\infty}(M_{q-1}\setminus \hat{\Sigma}_{q-1}\times [s,T])
\end{equation}for any $T>s>0$. According to Lemma \ref{lemma4.2}, one has
\begin{equation}\label{Linftymeancurv}
\begin{split}
&|\sqrt{-1}\Lambda_{\omega_{q-1,\epsilon}}F_{H_{q-1,\epsilon}(s')}|_{H_{q-1,\epsilon}(s')}(x)\\
\leq &\frac{C(C_S(\omega_{q-1,\epsilon}),n,\tilde{C}_1,T)}{s^{2n+1}}\int_{M_{q-1}\setminus\Sigma_{q-1}}|\sqrt{-1}\Lambda_{\omega_{q-1,\epsilon}}F_{H_0}|_{H_0}\frac{\omega_{q-1,\epsilon}^n}{n!}\\
\end{split}
\end{equation}for any $x\in M_{q-1}\setminus\Sigma_{q-1}$, $T>s>0$ and $s'\in[s,T]$. Fixing $(\epsilon_1,\epsilon_2,\cdots,\epsilon_{q-2})$ and running the same procedure, we obtain a solution of Donaldson's heat flow on $M_{q-2}\setminus \hat{\Sigma}_{q-2}$. Repeating the argument, we obtain a solution of Donaldson's heat flow on $M\setminus\Sigma_{\mathcal{E}}$ which satisfies
\begin{equation}
\int_{M\setminus\Sigma_{\mathcal{E}}}|\sqrt{-1}\Lambda_\omega F_{H(t)}|_{H(t)}\frac{\omega^n}{n!}\leq e^{2\tilde{C}_1(t-t_0)}\int_{M\setminus\Sigma_{\mathcal{E}}}|\sqrt{-1}\Lambda_\omega F_{H(t_0)}|_{H(t_0)}\frac{\omega^n}{n!},
\end{equation}
\begin{equation}
\begin{split}
&\int_{M\setminus\Sigma_{\mathcal{E}}}\ln(\mathrm{tr} h(t)+\mathrm{tr} h^{-1}(t))\frac{\omega^n}{n!}\\
\leq &e^{2\tilde{C}_1 t}\ln(2\mathrm{rank}(\mathcal{E}))\mathrm{Vol}(M,\omega)+\frac{e^{2\tilde{C}_1 t}-1}{2\tilde{C}_1}\int_{M\setminus\Sigma_{\mathcal{E}}}|\Lambda_{\omega}F_{H_0}|_{H_0}\frac{\omega^n}{n!},
\end{split}
\end{equation}
\begin{equation}
|\sqrt{-1}\Lambda_{\omega}F_{H}|_{L^{\infty}(M\setminus\Sigma_{\mathcal{E}}\times[s,T],H)}\leq\frac{C(C_S(\omega),n,\tilde{C}_1,T)}{s^{2n+1}}\int_{M\setminus\Sigma_{\mathcal{E}}}|\sqrt{-1}\Lambda_\omega F_{H_0}|_{H_0}\frac{\omega^n}{n!}
\end{equation}for any $t\geq t_0\geq 0$ and any $T>s>0$. This finishes the proof.
\end{proof}
\begin{prop}\label{proposition2.1}
Let $\mathcal{E}$ be a torsion-free sheaf with the singular set $\Sigma_{\mathcal{E}}$ over a compact Gauduchon manifold $(M,\omega)$, $S$ be a saturated subsheaf of $\mathcal{E}$ with the singular set $\Sigma_S$. Assume there is a sequence of connections $A_j\in \mathcal{A}_{H_0}^{1,1}$ on the Hermitian bundle ($\mathcal{E}|_{M\setminus\Sigma}$, $H_0$), where $\Sigma$ is a closed subset of $M$ with real Hausdorff codimension at least $4$ and satisfies $\Sigma_S\cup\Sigma_{\mathcal{E}}\subset \Sigma$, such that
$\\$$(1)$ $A_j\to A_{\infty}$ in $C_{loc}^{\infty}$-topology off $\Sigma$ as $j\to\infty$;
$\\$$(2)$ $A_j=g_j(A_0)$ for some complex gauge transformation $g_j$ and $\|\sqrt{-1}\Lambda_\omega F_{\tilde{H}_j}\|_{L^1(\omega,\tilde{H}_j)}$ is bounded uniformly in $j$, where the metric $\tilde{H}_j$ is defined by $H_0^{-1}\tilde{H}_j=g_j^{\ast_{H_0}}g_j$ and $F_{\tilde{H}_j}$ is the Chern curvature of $\tilde{H}_j$ on $(\mathcal{E},\overline{\partial}_{A_0})$;
$\\$$(3)$ There exist a sequence of blow-ups with smooth center: $\pi_i:M_i\to M_{i-1}, i=1,\cdots,q$ and an exact sequence of holomorphic vector bundles $0\to\tilde{S}\to\tilde{E}\to\tilde{W}\to 0$ over $M_q$ such that the composition $\pi=\pi_1\circ\pi_2\circ\cdots\circ\pi_q: M_q\to M$ is biholomorphic outside $\Sigma$, $\tilde{E}$ and $\tilde{S}$ are isomorphic to $\mathcal{E}$ and $S$ outside $\Sigma$ respectively. And the pullback geometric objects $\pi^{\ast}A_0$, $\pi^{\ast}\tilde{H}_j$, $\pi^{\ast}H_0$ can be extended smoothly on $\tilde{E}$ over $M_q$.

Let $i_0: S \hookrightarrow \mathcal{E}$ be the holomorphic inclusion. Then there is a subsequence of $\{g_j\circ i_0\}$, up to rescaling, converges to a non-zero holomorphic map $f_{\infty}:(S,\overline{\partial}_{A_0})\to(E_{\infty},\overline{\partial}_{A_{\infty}})$ in $C_{loc}^{\infty}$-topology off $\Sigma$ as $j\to \infty$.
\end{prop}
\begin{proof}

Let $H_0^S$ be the induced metric by $H_0$ on $S$. Based on Proposition \ref{lemma2.2}, we have the solutions $H_j(t)$ and $H^{S}(t)$ of the evolved equation (\ref{inequalty2.5}) with the initial conditions  $H_j(0)=\tilde{H}_j$ and $H^S(0)=H^S_0$ on $(\mathcal{E}|_{M\setminus\Sigma},\overline{\partial}_{A_0})$ and $(S|_{M\setminus\Sigma},\overline{\partial}_{A_0})$ respectively, satisfying
\begin{equation}
\begin{split}
&|\sqrt{-1}\Lambda_\omega F_{H_j(t)}|_{L^{\infty}(M\setminus\Sigma\times[s,T],H_j(t))}\\
\leq&\frac{C(C_S(\omega),n,\tilde{C}_1,T)}{s^{2n+1}}\int_{M\setminus\Sigma}|\sqrt{-1}\Lambda_\omega F_{H_j(0)}|_{H_j(0)}\frac{\omega^n}{n!}
\end{split}
\end{equation}and
\begin{equation}
\begin{split}
&|\sqrt{-1}\Lambda_\omega F_{H^S(t)}|_{L^{\infty}(M\setminus\Sigma\times[s,T],H^S(t))}\\
\leq&\frac{C(C_S(\omega),n,\tilde{C}_1,T)}{s^{2n+1}}\int_{M\setminus\Sigma}|\sqrt{-1}\Lambda_\omega F_{H^S(0)}|_{H^S(0)}\frac{\omega^n}{n!}
\end{split}
\end{equation}for any $T>s>0$.  Calculating directly shows that
\begin{equation}\label{prop02ineq2}
\frac{\partial}{\partial t}|i_0|^2_{H^S(t),H_j(t)}=-2\langle\sqrt{-1}\Lambda_\omega F_{H_j(t)}\circ i_0,i_0\rangle+2\langle i_0\circ\sqrt{-1}\Lambda_\omega F_{H^S(t)},i_0\rangle,
\end{equation}and
\begin{equation}
(2\sqrt{-1}\Lambda_\omega\partial\overline{\partial}-\frac{\partial}{\partial t})|i_0|^2_{H^S(t),H_j(t)}\geq 0.
\end{equation}
By the same method as in the proof of (\ref{inequalty2.8}), it follows that
\begin{equation}
|i_0|^2_{H^S,H_j}\in L^{\infty}(M\setminus\Sigma\times[s,T]),
\end{equation}and
\begin{equation}
\int_{M\setminus\Sigma}|i_0|^2_{H^S(t),H_j(t)}\frac{\omega^n}{n!}\leq e^{2\tilde{C}_1(t-t_0)}\int_{M\setminus\Sigma}|i_0|^2_{H^S(t_0),H_j(t_0)}\frac{\omega^n}{n!}
\end{equation}for any $T>s>0$ and $t_0+1>t\geq t_0> 0$. Applying Lemma \ref{lemma4.2}, we obtain
\begin{equation}\label{0.71}
|i_0|^2_{H^S(t_0+\hat{s}),H_j(t_0+\hat{s})}(x)\leq \frac{C(C_S(\omega),n,t_0+1,\tilde{C}_1)}{\hat{s}^{2n+1}}\int_{M\setminus\Sigma}|i_0|^2_{H^S(t_0),H_j(t_0)}\frac{\omega^n}{n!}
\end{equation}for $x\in M\setminus\Sigma$ and any $0<\hat{s}<1$. By the uniform $L^1$-bound in the assumptions, we derive
\begin{equation}\label{prop02ineq1}
2(|\Lambda_\omega F_{H_j(t)}|_{H_j(t)}+|\Lambda_\omega F_{H^S(t)}|_{H^S(t)})(x)\leq C_2
\end{equation}for all $x\in M\setminus\Sigma$ and $t_0+1>t\geq t_0 > 0$ with $t_0$ fixed, where $C_2=C_2(t_0)$ is a uniform constant independent of $j$. Combining (\ref{prop02ineq1}) and (\ref{prop02ineq2}) yields that
\begin{equation}
-C_2\leq\frac{\partial}{\partial t}\ln |i_0|^2_{H^S(t),H_j(t)}(x)\leq C_2
\end{equation}for all $x\in M\setminus\Sigma$ and any $t$ with $t_0+1>t\geq t_0> 0$. Then
\begin{equation}\label{0.70}
|i_0|_{H^S(t_0),H_j(t_0)}^2(x)\leq e^{C_2\hat{s}}|i_0|^2_{H^S(t_0+\hat{s}),H_j(t_0+\hat{s})}(x)
\end{equation}for all $x\in M\setminus\Sigma$ and any $0<\hat{s}<1$. From (\ref{0.71}) and (\ref{0.70}), one can see
\begin{equation}\label{ineq1107}
|i_0|_{H^S(t_0),H_j(t_0)}^2(x)\leq e^{C_2\hat{s}}\frac{C(C_S(\omega),n,t_0+1,\tilde{C}_1)}{\hat{s}^{2n+1}}\int_{M\setminus\Sigma}|i_0|^2_{H^S(t_0),H_j(t_0)}\frac{\omega^n}{n!}
\end{equation}for $x\in M\setminus\Sigma$ and any $0<\hat{s}<1$.

Define the holomorphic map $\hat{f}_j:(S|_{M\setminus\Sigma},\overline{\partial}_{A_0})\to(\mathcal{E}|_{M\setminus\Sigma},\overline{\partial}_{A_j})$ by $\hat{f}_j=g_j\circ i_0$. It is easy to check that
\begin{equation}
|\hat{f}_j|_{H^S_0,H_0}=|i_0|_{H^S_0,\tilde{H}_j}.
\end{equation}Set
\begin{equation}
f_j=(\int_M |i_0|^2_{H^S(1),H_j(1)}\frac{\omega^n}{n!})^{-\frac{1}{2}}\hat{f}_j.
\end{equation}Then (\ref{ineq1107}) means that there is a constant $C_3$ such that
\begin{equation}\label{ineq11072}
\sup_{x\in M\setminus\tilde{\Sigma}}(\int_M |i_0|^2_{H^S(1),H_j(1)}\frac{\omega^n}{n!})^{-1}|i_0|^2_{H^S(1),H_j(1)}(x)\leq C_3
\end{equation}for all $j$. Set $h_j(t)=\tilde{H}_j^{-1}H_j(t)$. By Proposition \ref{lemma2.2}, we also get
\begin{equation}\label{0.79}
\begin{split}
&\int_{M\setminus\Sigma}\ln(\mathrm{tr} h_j(t)+\mathrm{tr} h_j^{-1}(t))\frac{\omega^n}{n!}\\
\leq &e^{2\tilde{C}_1t}\ln(2\mathrm{rank}(\mathcal{E}))\mathrm{Vol}(M,\omega)+\frac{e^{2\tilde{C}_1 t}-1}{2\tilde{C}_1}\int_{M\setminus\Sigma}|\Lambda_\omega F_{\tilde{H}_j}|_{\tilde{H}_j}\frac{\omega^n}{n!}.\\
\end{split}
\end{equation}On the other hand, it is not hard to verify that
\begin{equation}\label{0.80}
2\sqrt{-1}\Lambda_\omega\partial\overline{\partial}\ln(\mathrm{tr} h_j(t)+\mathrm{tr} h_j^{-1}(t))\geq-2|\Lambda_\omega F_{H_j(t)}|_{H_j(t)}-2|\Lambda_\omega F_{\tilde{H}_j}|_{\tilde{H}_j}
\end{equation}on $M\setminus\Sigma$ for all $t>0$.

For any compact subset $\Omega\subset M\setminus\Sigma$, set $\mathrm{dist}(\Omega)=\inf\{d(x,y)\,|\,x\in\Omega,\,y\in\Sigma\}>0$, where $d$ is the distance function on $(M,\omega)$. Let $B=\cup_{y\in\Sigma} B^{\omega}_{y}(\frac{1}{2}\mathrm{dist}(\Omega))$ and $\Omega'=M\setminus B$. Then we choose the cut-off function $\psi$ such that $\psi\equiv 1$ on $\Omega$, $\psi\equiv 0$ on $B$, and $|\mathrm{d}\psi|_{\omega}\leq\frac{4}{\mathrm{dist}(\Omega)}$. By the assumption, $A_j$ are locally bounded in $C^k$ for all $k\geq 0$ outside $\Sigma$, then there holds
\begin{equation}\label{0.81}
|\Lambda_\omega F_{\tilde{H}_j}|_{\tilde{H}_j}=|\Lambda_\omega F_{A_j}|_{H_0}\leq C_4
\end{equation}on $\Omega'$, where $C_4$ is a constant independent of $j$. Using (\ref{prop02ineq1}),(\ref{0.79}),(\ref{0.80}),(\ref{0.81}), the cut-off function $\psi$ and Moser's iteration, we deduce
\begin{equation}\label{0.82}
\sup_{\Omega}\ln(\mathrm{tr} h_j(1)+\mathrm{tr} h_j^{-1}(1))\leq C_5\int_{M\setminus\Sigma}\ln(\mathrm{tr} h_j(1)+\mathrm{tr} h_j^{-1}(1))\frac{\omega^n}{n!}\leq C_6,
\end{equation}where $C_5$, $C_6$ are constants independent of $j$. In a similar way, one can obtain the local $C^0$-estimate of $H^S(1)$, i.e. for any compact subset $\Omega\subset\subset M\setminus\Sigma$, there exists a constant $C_7$ such that
\begin{equation}\label{0.83}
\sup_{\Omega}\ln(\mathrm{tr}((H_0^S)^{-1}H^S(1))+\mathrm{tr}((H^S(1))^{-1}H_0^S))\leq C_7.
\end{equation}For any compact subset $\Omega\subset M\setminus\Sigma$, by (\ref{ineq11072}),(\ref{0.82}),(\ref{0.83}), we get
\begin{equation}
\sup_{x\in\Omega}|f_j|^2_{H_0^S,H_0}(x)\leq C(\Omega),
\end{equation}where $C(\Omega)$ is a constant independent of $j$.

 Because $f_j$ is $\overline{\partial}_{A_0,A_j}$-holomorphic, by the above uniform local $C^0$-bound of $f_j$ and the assumption that $A_j\to A_{\infty}$ in $C^{\infty}_{loc}$-topology outside $\Sigma$ as $j\to\infty$, the standard elliptic theory implies that there exists a subsequence of $f_j$ (also denoted by $f_j$ for simplicity) such that $f_j$ converges to a holomorphic map $f_{\infty}:(S,\overline{\partial}_{A_0})\to(E_{\infty},\overline{\partial}_{A_{\infty}})$ in $C^{\infty}_{loc}$-topology outside $\Sigma$ as $j\to\infty$.

Now we only need to prove that $f_\infty$ is non-zero. Since $\Sigma$ is of Hausdorff codimension at least $4$, for any small $\bar{\delta} >0$ satisfying $\bar{\delta} C_3<1$, we can choose a compact subset $\Omega_{\bar{\delta}}\subset M\setminus\Sigma$ such that
\begin{equation}\label{0.85}
\int_{M\setminus\Omega_{\bar{\delta}}}\frac{\omega^n}{n!}\leq\bar{\delta}.
\end{equation}Clearly, the local uniform estimates (\ref{0.82}) and (\ref{0.83}) give us that there is a positive constant $C_{\overline{\delta}}$ such that
\begin{equation}
C_{\overline{\delta}}^{-1}|i_0|^2_{H^S(1),H_j(1)}(x)\leq|i_0|^2_{H^S_0,\tilde{H}_j}(x)\leq C_{\overline{\delta}}|i_0|^2_{H^S(1),H_j(1)}(x)
\end{equation}for all $x\in \Omega_{\overline{\delta}}$ and $j$. Then
\begin{equation}
\begin{split}
\int_{\Omega_{\overline{\delta}}}|f_\infty|^2_{H^S_0,H_0}\frac{\omega^n}{n!}&=\lim_{j\to +\infty}\int_{\Omega_{\overline{\delta}}}|f_j|^2_{H^S_0,H_0}\frac{\omega^n}{n!}\\
&=\lim_{j\to +\infty}(\int_M |i_0|^2_{H^S(1),H_j(1)}\frac{\omega^n}{n!})^{-1}\int_{\Omega_{\overline{\delta}}}|i_0|^2_{H^S_0,\tilde{H}_j}\frac{\omega^n}{n!}\\
&\geq \lim_{j\to +\infty}C_{\overline{\delta}}^{-1}(\int_M |i_0|^2_{H^S(1),H_j(1)}\frac{\omega^n}{n!})^{-1}\int_{\Omega_{\overline{\delta}}}|i_0|^2_{H^S(1),H_j(1)}\frac{\omega^n}{n!}\\
&\geq C_{\overline{\delta}}^{-1}(1-\overline{\delta}C_3)>0.\\
\end{split}
\end{equation}Therefore, $f_\infty$ is a non-zero holomorphic map.

\end{proof}

\subsection{Proof of Theorem \ref{thm1.9}}

	Let $\{ \mathcal{E}_{\alpha,\beta}^{}\}$ be the HNS-filtration of the holomorphic vector bundle $(E,{\overline{\partial}}_{{A}_{0}^{}}^{})$, the associated graded object ${Gr}_{\omega}^{HNS}(E,\overline{\partial}_{A_0})={\oplus}_{\alpha=1}^l{\oplus}_{\beta=1}^{l_{\alpha}^{}}\mathcal{Q}_{\alpha,\beta}^{}$ can be uniquely determined by the isomorphism class of $(E , \overline{\partial}_{A_0})$, where $\mathcal{Q}_{\alpha,\beta}^{}=\mathcal{E}_{\alpha,\beta}^{}/\mathcal{E}_{\alpha,\beta-1}^{}$. We refer to ${\Sigma}_{alg}$ as the singular set of the HNS-filtration $\{\mathcal{E}_{\alpha,\beta}^{}\}$, which is a complex analytic subset of complex codimension at least 2. Based on Theorem \ref{consistence of HN type}, Lemma \ref{lemma3.4} and  Proposition \ref{proposition2.1}, we can prove Theorem \ref{thm1.9}. Our argument is similar as that in \cite[Section 5]{LZ1} and we include it here for completeness. We first give a sketch of the proof.

Set  $\mathcal{Q}^{\alpha,\beta}=E/\mathcal{E}_{\alpha,\beta-1}^{}$, then one can obtain a sequence of torsion-free sheaves $\{ \mathcal{Q}^{\alpha,\beta}\}$. We will use induction to verify that $\mathcal{Q}^{\alpha,\beta}$ satisfies the following three properties for each $\alpha$ and $\beta$.

{\bf Inductive hypotheses}: There is a sequence of connections $A_j^{\mathcal{Q}^{\alpha,\beta}}\in \mathcal{A}_{\mathcal{Q}^{\alpha,\beta},H_0}^{1,1}$ on the Hermitian bundle ($\mathcal{Q}^{\alpha,\beta}|_{M\setminus (\Sigma_{alg}\cup\Sigma_{an})}$, $H_0$) such that\\
 $(1)$ $A_j^{\mathcal{Q}^{\alpha,\beta}}\to A_{\infty}^{\mathcal{Q}^{\alpha,\beta}_{\infty}}$ in $C_{loc}^{\infty}$-topology off $\Sigma_{alg}\cup\Sigma_{an}$, where  ${\Sigma}_{alg}$ is the singular set of the HNS-filtration $\{\mathcal{E}_{\alpha,\beta}^{}\}$ and $\Sigma_{an}$ is the analytic bubbling set in Theorem \ref{thm1.1};
$\\$$(2)$ $ A_j^{\mathcal{Q}^{\alpha,\beta}}=g_j(A_0^{\mathcal{Q}^{\alpha,\beta}})$ for some complex gauge transformation $g_j$ and $\|\sqrt{-1}\Lambda_\omega(F_{A^{\mathcal{Q}^{\alpha,\beta}}_{H_j}})\|_{L^1(\omega,H_j)}$ is uniformly bounded in $j$, where the metric $H_j$ is defined by $H_0^{-1}H_j=g_j^{\ast_{H_0}}g_j$ and $A^{\mathcal{Q}^{\alpha,\beta}}_{H_j}$ is the Chern connection of $H_j$ on $(\mathcal{Q}^{\alpha,\beta},\overline{\partial}_{A_0^{\mathcal{Q}^{\alpha,\beta}}})$;
                                  $\\$$(3)$ $(\mathcal{Q}^{\alpha,\beta}_{\infty},\overline{\partial}_{A_{\infty}^{\mathcal{Q}^{\alpha,\beta}_{\infty}}})$ has the same HN-type as that of $(\mathcal{Q}^{\alpha,\beta},\overline{\partial}_{A_0^{\mathcal{Q}^{\alpha,\beta}}})$.

Once $\mathcal{Q}^{\alpha,\beta}$ satisfies the inductive hypotheses, by Proposition \ref{proposition2.1}, we can construct a non-zero holomorphic map  $$f_{\alpha,\beta}:(\mathcal{Q}_{\alpha,\beta},\overline{\partial}_{A_0^{\mathcal{Q}_{\alpha,\beta}}}) \to (E_\infty,\overline{\partial}_{A_{\infty}})$$ off $\Sigma_{alg}\cup\Sigma_{an}$. Utilizing Theorem \ref{consistence of HN type}, we can prove that the holomorphic map
\begin{equation}
{\oplus}_{\alpha=1}^l{\oplus}_{\beta=1}^{l_{\alpha}}f_{\alpha,\beta}: {Gr}_{\omega}^{HNS}(E,\overline{\partial}_{A_0}) \to (E_\infty,\overline{\partial}_{A_{\infty}})
\end{equation}is an isomorphism off $\Sigma_{alg}\cup\Sigma_{an}$. By Hartogs' extension theorem, the above map can extend over $\Sigma_{alg}\cup\Sigma_{an}$ as a holomorphic isomorphism from $(E_\infty,\overline{\partial}_{A_{\infty}})$ to $Gr_{\omega}^{HNS}(E,\overline{\partial}_{A_0})^{\ast\ast}$.

Now we give the details of the proof.

 \begin{proof}[Proof of Theorem \ref{thm1.9}]
We first show $\mathcal{Q}^{1,1}$ (i.e. $E$) satisfies the inductive hypotheses $(1)-(3)$.
Let $H(t)$ be the long time solution of the Hermitian-Yang-Mills flow (\ref{Flow}) on the holomorphic vector bundle $(E,\overline{\partial}_{A_0})$ with the initial metric $H_0$, and $A(t)$ be the solution of the modified Yang-Mills flow (\ref{modified flow}) on the Hermitian vector bundle $(E,H_0)$ with the initial connection $A_0$. Then we have $A(t)=\sigma(t)(A_0)$, where $\sigma(t)$ satisfies $\sigma^{\ast_{H_0}}(t)\sigma(t)=h(t)=H_0^{-1}H(t)$.

By Theorem \ref{thm1.1}, there is a sequence of connections $\{A_{j}\}$ which converge to $A_\infty$ in $C^{\infty}_{loc}$-topology outside $\Sigma_{an}$ as $j\to\infty$, where $A_{j}=a_j(A(t_j))=a_j \circ \sigma(t_j)(A_0)$ for some unitary gauge transformation $a_j$. The same calculation as (\ref{eq2.20}) shows that
\begin{equation}\label{finalequa1108}
(2\sqrt{-1}\Lambda_{\omega}\partial\overline{\partial}-\frac{\partial}{\partial t})(|\sqrt{-1}\Lambda_{\omega}F_{H(t)}|_{H(t)}^2+\zeta)^{\frac{1}{2}}\geq 0,
\end{equation}
for any $\zeta> 0$. Integrating both sides of (\ref{finalequa1108}) over $(M,\omega)$ and letting $\zeta\to 0$, one can directly deduce that
\begin{equation}\label{ineq20241108}
\|\sqrt{-1}\Lambda_\omega F_{H(t_1)}\|_{L^1(\omega,H(t_1))}\leq \|\sqrt{-1}\Lambda_\omega F_{H(t_2)}\|_{L^1(\omega,H(t_2))}
\end{equation}for any $t_1\geq t_2 \geq 0$. Using (\ref{ineq20241108}) and Theorem \ref{consistence of HN type}, one can easily check that $\mathcal{Q}^{1,1}$ satisfies the inductive hypotheses $(1)$, $(2)$, $(3)$.

Let $\mathcal{S}:=\mathcal{E}_{1,1}$ be the first stable coherent subsheaf corresponding to the HNS-filtration of the holomorphic vector bundle $(E,\overline{\partial}_{A_0})$ and $i_0:  \mathcal{S} \to (E,\overline{\partial}_{A_0})$ be the holomorphic inclusion. By Sibley's result on the resolution of singularities of the HNS-filtration (\cite[Proposition 4.3]{Sib}), there is a finite sequence of blowups along complex submanifolds, whose composition $\tilde{\pi}:\tilde{M}\to M$ enjoys the following properties: there is a filtration
\begin{equation}
0=E_{0,0}\subset E_{1,1}\subset \cdots \subset E_{l,k_l}=\tilde{E}:=\tilde{\pi}^{\ast}E
\end{equation}by subbundles, and $E_{\alpha,\beta}$ is isomorphic to $\mathcal{E}_{\alpha,\beta}$ outside $\tilde{\pi}^{-1}(\Sigma_{alg})$ for each $\alpha$ and $\beta$.

According to Proposition \ref{proposition2.1}, there is a subsequence of $f_j=a_j\circ \sigma(t_j)\circ i_0$, up to rescaling, converging to a non-zero holomorphic map $f_{\infty}: \mathcal{S}\to (E_{\infty},\overline{\partial}_{A_{\infty}})$ outside $\Sigma_{alg}\cup\Sigma_{an}$ as $j\to\infty$. Since the Hausdorff codimension of $\Sigma_{alg}\cup\Sigma_{an}$ is at least $4$, by Hartogs' extension theorem (also see \cite[Lemma 3]{Shiffman68} for this condition), $f_{\infty}$ extends to a sheaf homomorphism $f_{\infty}:\mathcal{S}\to(E_{\infty},\overline{\partial}_{A_{\infty}})$ on the whole $M$.

Denote by $\xi^{j}_1$ the orthogonal projection onto $a_j\circ\sigma(t_j)(\mathcal{S})$ with respect to $H_0$. Applying Lemma \ref{lemma3.4}, by Uhlenbeck and Yau's regularity statement of $L^2_1$-subbundles (\cite{UY86}), we know that $\xi_1^j$ converges to $\xi_1^{\infty}$ strongly in ${L}^{p}\cap{L}_{1,loc}^{2}$ off ${\Sigma}_{an}$ as $j\rightarrow +\infty$, and $\xi_1^{\infty}$ determines a subsheaf $\mathcal{E}_{1,1}^\infty$ of $(E_{\infty},\overline{\partial}_{A_{\infty}})$ with $\rank(\mathcal{E}_{1,1}^\infty)=\rank(\mathcal{S})$ and $\mu_{\omega}(\mathcal{E}_{1,1}^\infty)=\mu_{\omega}(\mathcal{S})$. On the other hand, the fact that $\xi_1^j\circ f_j=f_j$ implies that $\xi_1^{\infty}\circ f_\infty=f_\infty$ and then $f_{\infty}: \mathcal{S} \to \mathcal{E}_{1,1}^\infty$. Moreover, Theorem \ref{consistence of HN type} states that $(E_{\infty},\overline{\partial}_{A_{\infty}})$ and $(E,\overline{\partial}_{A_0})$ have the same HN-type, and so $\mathcal{E}_{1,1}^\infty$ is $\omega$-semistable. Notice that $\mathcal{S}$ is $\omega$-stable. Due to \cite[$\mathrm{V}$.7.11,7.12]{Kobayashi}, the nonzero holomorphic map $f_{\infty}$ must be isomorphisic, i.e.
\begin{equation}
 \mathcal{S}\cong \mathcal{E}_{1,1}^{\infty},
\end{equation}
and $\mathcal{E}_{1,1}^{\infty}$ is an $\omega$-stable subsheaf of $(E_{\infty},\overline{\partial}_{A_{\infty}})$. This is the first step in the induction.

Now we are going to use the fact $\mathcal{Q}^{1,1}$ satisfies $(1)-(3)$ to infer that $\mathcal{Q}^{1,2}$ also satisfies the inductive hypotheses. For simplicity, denote $\mathcal{Q}^{1,2}=E/\mathcal{S}$ as $\mathcal{Q}$. Then we have $Gr_\omega^{HNS}(E,\overline{\partial}_{A_0})=\mathcal{S}\oplus Gr_\omega^{HNS}(\mathcal{Q},\overline{\partial}_{A_0^\mathcal{Q}})$. Since $H_\infty$ is a direct sum of admissible Hermitian-Einstein metrics, there is a holomorphic orthogonal decomposition $(E_{\infty},\overline{\partial}_{A_{\infty}})=\mathcal{E}_{1,1}^{\infty}\oplus \mathcal{Q}_{\infty}$, where $\mathcal{Q}_{\infty}=(\mathcal{E}_{1,1}^{\infty})^{\perp}$ on $M\setminus(\Sigma_{alg}\cup\Sigma_{an})$.

Let $\{e_a\}$ be a local frame of $\mathcal{S}$, and $H_{j,a\overline{b}}=\langle f_j(e_a),f_j(e_b)\rangle_{H_0}$. One can write the orthogonal projection $\xi_{1}^{j}$ as
\begin{equation}
     \xi_1^{j}(X)=\langle X,f_j(e_b)\rangle_{H_0} H_{j}^{a\overline{b}}f_j(e_a)
\end{equation}for any $X\in \Gamma(E)$, where $(H_j^{a\overline{b}})$ is the inverse of the matrix $(H_{j,a\overline{b}})$. Because $f_j\to f_{\infty}$ in $C_{loc}^{\infty}$-topology off $\Sigma_{alg}\cup\Sigma_{an}$ as $j\to\infty$ and $f_\infty$ is injective, we can show that $\xi_{1}^{j}\to\xi_{1}^{\infty}$ in $C_{loc}^{\infty}$-topology off $\Sigma_{alg}\cup\Sigma_{an}$ as $j\to\infty$. Let $\hat{\xi}_1: E_{\infty}\to E_{\infty}$ be the orthogonal projection onto $\mathcal{E}_{1,1}^{\infty}$ with respect to $H_\infty$. By \cite[Lemma 5.12]{daskalopoulos1992topology}, we can choose a sequence of unitary gauge transformations $\{u_j\}$ such that $\xi_1^{j}=u_j\circ \hat{\xi}_1 \circ u_j^{-1}$ and $u_j\to \mathrm{Id}$ in $C_{loc}^{\infty}$-topology on $M\setminus(\Sigma_{alg}\cup\Sigma_{an})$ as $j\to\infty$. It is easy to verify that $u_j(\mathcal{Q}_{\infty})=(\xi_1^{j}(E))^{\perp}$. Note that there are the bundle isomorphism $p^{\ast_{H_0}}:\mathcal{Q}\to \mathcal{S}^{\perp}$ and the unitary gauge transformation $u_0:\mathcal{Q}_{\infty}\to \mathcal{S}^{\perp}$. Consider the induced connections on $\mathcal{Q}$
\begin{equation}
     D_{A_j^\mathcal{Q}}=(p^{\ast_{H_0}})^{-1}\circ a_0^{-1}\circ u_0\circ \hat{\xi}_1^{\perp}\circ u_j^{-1}\circ a_j\circ D_{A(t_j)}\circ a_j^{-1}\circ u_j\circ \hat{\xi}_1^{\perp}\circ u_0^{-1}\circ a_0\circ p^{\ast_{H_0}},
\end{equation}and the complex gauge transformation of $\mathcal{Q}$
\begin{equation}
         h_j=(p^{\ast_{H_0}})^{-1}\circ a_0^{-1}\circ u_0\circ \hat{\xi}_1^{\perp}\circ u_j^{-1}\circ a_j \circ \sigma(t_j) \circ p^{\ast_{H(t_j)}}.
\end{equation}Then
\begin{equation}
\overline{\partial}_{A_j^\mathcal{Q}}=h_j\circ \overline{\partial}_{A_0^\mathcal{Q}} \circ h_j^{-1}
\end{equation}and
\begin{equation}
\partial_{A_j^\mathcal{Q}}=(h_j^{\ast_{H_0}})^{-1}\circ \partial_{A_0^\mathcal{Q}} \circ h_j^{\ast_{H_0}}.
\end{equation}
From the definition, one can directly deduce that $((p^{\ast_{H_0}})^{-1}\circ a_0^{-1}\circ u_0)^{\ast}(A_j^\mathcal{Q}) \to A_{\infty}^{\mathcal{Q}_{\infty}}$ in $C_{loc}^{\infty}$-topology on $M\setminus(\Sigma_{alg}\cup\Sigma_{an})$, as $j\to\infty$. Hence $\mathcal{Q}$ satisfies the inductive hypothesis $(1)$. Meanwhile, straightforward calculation shows that $h_j^{\ast_{H_0}}h_j=(H_0^{\mathcal{Q}})^{-1}H^{\mathcal{Q}}(t_j)$, where $H^{\mathcal{Q}}(t)$ denotes the induced metric on the quotient $\mathcal{Q}$ by $H(t)$. Then we obtain
\begin{equation}
\begin{split}
\|\sqrt{-1}\Lambda_\omega F_{H^{\mathcal{Q}}(t_j)}\|_{L^1(\omega,H^{\mathcal{Q}}(t_j))}&=\|\sqrt{-1}\Lambda_\omega F_{A^{\mathcal{Q}_{t_j}}(t_j)}\|_{L^1(\omega,H_0)}\\
&=\|\sqrt{-1}\Lambda_\omega F_{A_j^{\mathcal{Q}_j}}\|_{L^1(\omega,H_0)},\\
\end{split}
\end{equation}where $\mathcal{Q}_{t_j}=\sigma(t_j)(\mathcal{Q})$, $\mathcal{Q}_j=a_j\circ\sigma(t_j)(\mathcal{Q})$.

Now let's consider the Gauss-Codazzi equation on $\mathcal{Q}_j$ with respect to the metric $H_0$ and the Chern connection $A_j$:
\begin{equation}
    F_{A_j^{\mathcal{Q}_j}}=(\xi_1^{j})^{\perp}\circ F_{A_j}\circ (\xi_1^{j})^{\perp}+\partial_{A_j} \xi_1^{j}\wedge\overline{\partial}_{A_j}\xi_1^{j}.
\end{equation}Then we have
\begin{equation}
\begin{split}
    &\int_{M\setminus(\Sigma_{alg})} |\sqrt{-1}\Lambda_{\omega}F_{A_j^{\mathcal{Q}_j}}|_{H_0}\frac{\omega^n}{n!}\\
    =&\int_{M\setminus(\Sigma_{alg})} |(\xi_1^{j})^{\perp}\sqrt{-1}\Lambda_{\omega}F_{A_j}(\xi_1^{j})^{\perp}+\sqrt{-1}\Lambda_{\omega}(\partial_{A_j} \xi_1^{j}\wedge\overline{\partial}_{A_j}\xi_1^{j})|_{H_0}\frac{\omega^n}{n!}\\
    \leq & \int_{M\setminus(\Sigma_{alg})} \{|\sqrt{-1}\Lambda_{\omega}F_{A_j}|_{H_0}+|\overline{\partial}_{A_j}\xi_1^{j}|^2_{H_0}\} \frac{\omega^n}{n!}.\\
\end{split}
\end{equation}By the degree formula of $\mathcal{S}$ (\cite[Lemma 3.2]{SIM}) and Gauss-Codazzi equation, we obtain that $\int_{M\setminus(\Sigma_{alg})} |\overline{\partial}_{A_j}\xi_1^{j}|^2_{H_0} \frac{\omega^n}{n!}$ is uniformly bounded. Then it implies that $\int_{M\setminus(\Sigma_{alg})} |\sqrt{-1}\Lambda_{\omega}F_{A_j^{\mathcal{Q}_j}}|_{H_0}\frac{\omega^n}{n!}$ is uniformly bounded. So the inductive hypothesis $(2)$ is satisfied. Obviously Theorem \ref{consistence of HN type} tells us that the inductive hypothesis $(3)$ is also valid. Hence $\mathcal{Q}$ satisfies the inductive hypotheses $(1)-(3)$.


Repeating the above argument by induction, we conclude
\begin{equation}
E_{\infty} \cong Gr_{\omega}^{HNS}(E,\overline{\partial}_{A_0})=\oplus_{\alpha=1}^{l}\oplus_{\beta=1}^{k_\alpha}\mathcal{Q}_{\alpha,\beta}
\end{equation}on $M\setminus(\Sigma_{alg}\cup\Sigma_{an})$. By Hartogs' extension theorem and the normality of reflexive sheaves, there exists a sheaf isomorphism
\begin{equation}
f:  (E_{\infty},\overline{\partial}_{A_{\infty}}) \to Gr_{\omega}^{HNS}(E,\overline{\partial}_{A_0})^{\ast\ast}
\end{equation}on $M$.
We are done.
     \end{proof}



\end{document}